\documentclass[reqno,12pt]{amsart}
\usepackage[a4paper, margin=1in]{geometry}
\makeatletter
\@namedef{subjclassname@2020}{%
  \textup{2020} Mathematics Subject Classification}
\makeatother
\title[Weak weak approximation for degree-two del Pezzo surfaces]{Weak weak approximation and the Hilbert property for degree-two del Pezzo surfaces}

\author{Julian Demeio}
\address{Julian Demeio \\
Department of Mathematical Sciences \\
University of Bath \\
North Road \\
Bath\\
BA2 7AY\\
UK.}
\email{jd2724@bath.ac.uk}

\author{Sam Streeter}
\address{Sam Streeter \\
School of Mathematics \\
University of Bristol \\
Woodland Road \\
Bristol \\
BS1 8UG \\
UK.}
\email{sam.streeter@bristol.ac.uk}

\author{Rosa Winter}
\address{Rosa Winter \\
Faculty of Mathematics and Computer Science \\
UniDistance Suisse \\
Schinerstrasse 18 \\
3900 Brig\\
Switzerland}
\email{rosa.winter@unidistance.ch}
\usepackage{amsmath}
\usepackage{amssymb}
\usepackage{amsthm}
\usepackage{enumitem}
\usepackage{hyperref}
\usepackage{tikz-cd}
\usepackage{color,soul}
\usepackage{mathrsfs}
\theoremstyle{definition}
\newtheorem{mydef}{Definition}[section]
\newtheorem{note}[mydef]{Note}
\newtheorem{remark}[mydef]{Remark}
\newtheorem{question}[mydef]{Question}
\newtheorem*{proofmainthm}{Proof of Theorem \ref{thm:main1}}
\theoremstyle{plain}
\newtheorem{theorem}[mydef]{Theorem}
\newtheorem{proposition}[mydef]{Proposition}
\newtheorem{lemma}[mydef]{Lemma}
\newtheorem{corollary}[mydef]{Corollary}
\newtheorem{conjecture}[mydef]{Conjecture}
\newcounter{tmp}
\let\originalleft\left
\let\originalright\right
\renewcommand{\left}{\mathopen{}\mathclose\bgroup\originalleft}
\renewcommand{\right}{\aftergroup\egroup\originalright}

\renewcommand{\P}{\mathbb{P}}
\newcommand{\Z}{\mathbb{Z}}

\newcommand{\F}{\mathbb{F}}

\DeclareMathOperator{\Br}{Br}
\DeclareMathOperator{\ch}{char}

\DeclareMathOperator{\Div}{Div}

\DeclareMathOperator{\Pic}{Pic}

\DeclareMathOperator{\Spec}{Spec}

    \DeclareFontFamily{U}{wncy}{}
    \DeclareFontShape{U}{wncy}{m}{n}{<->wncyr10}{}
    \DeclareSymbolFont{mcy}{U}{wncy}{m}{n}
    \DeclareMathSymbol{\Sh}{\mathord}{mcy}{"58} 
\makeatletter
\newcommand*{\da@rightarrow}{\mathchar"0\hexnumber@\symAMSa 4B }
\newcommand*{\da@leftarrow}{\mathchar"0\hexnumber@\symAMSa 4C }
\newcommand*{\xdashrightarrow}[2][]{%
  \mathrel{%
    \mathpalette{\da@xarrow{#1}{#2}{}\da@rightarrow{\,}{}}{}%
  }%
}
\newcommand{\xdashleftarrow}[2][]{%
  \mathrel{%
    \mathpalette{\da@xarrow{#1}{#2}\da@leftarrow{}{}{\,}}{}%
  }%
}
\newcommand*{\da@xarrow}[7]{%
\sbox0{$\ifx#7\scriptstyle\scriptscriptstyle\else\scriptstyle\fi#5#1#6\m@th$}%
 \sbox2{$\ifx#7\scriptstyle\scriptscriptstyle\else\scriptstyle\fi#5#2#6\m@th$}%
  \sbox4{$#7\dabar@\m@th$}%
  \dimen@=\wd0 %
  \ifdim\wd2 >\dimen@
    \dimen@=\wd2 %
  \fi
  \count@=2 %
  \def\da@bars{\dabar@\dabar@}%
  \@whiledim\count@\wd4<\dimen@\do{%
    \advance\count@\@ne
    \expandafter\def\expandafter\da@bars\expandafter{%
      \da@bars
      \dabar@ 
    }%
  }%
  \mathrel{#3}%
  \mathrel{%
    \mathop{\da@bars}\limits
    \ifx\\#1\\%
    \else
      _{\copy0}%
    \fi
    \ifx\\#2\\%
    \else
      ^{\copy2}%
    \fi
  }%
  \mathrel{#4}%
}
\makeatother
\begin{document}
\begin{abstract}
We prove that del Pezzo surfaces of degree $2$ over a field $k$ satisfy weak weak approximation if $k$ is a number field and the Hilbert property if $k$ is Hilbertian of characteristic zero, provided that they contain a $k$-rational point lying neither on any $4$ of the $56$ exceptional curves nor on the ramification divisor of the anticanonical morphism. This builds upon results of Manin, Salgado--Testa--V\'arilly-Alvarado, and Festi--van Luijk on the unirationality of such surfaces, and upon work of the first two authors verifying weak weak approximation under the assumption of a conic fibration.
\end{abstract}
\maketitle
\setcounter{tocdepth}{1}
\tableofcontents
\section{Introduction}
A central belief in arithmetic geometry is that, if a Fano variety has one rational point, then it has many, and that they are well-distributed. Generally, the difficulty of verifying this belief increases with dimension. For curves, the situation is simple: a smooth projective curve is Fano if and only if it has genus $0$, in which case it is isomorphic to $\mathbb{P}^1$ as soon as it has a rational point. For varieties of dimension at least three, little is known other than results via the circle method; the two-dimensional case, that of \emph{del Pezzo surfaces}, has been a major area of research and progress in the last fifty years. 

One property capturing the abundance and distribution of rational points is \emph{weak approximation}, satisfied by a variety $X$ over a number field $K$ if its set $X\left(K\right)$ of rational points is dense in $\prod_{v \in \Omega_K}X\left(K_v\right)$, where $\Omega_K$ is the set of places of $K$. Many varieties fail weak approximation, but this failure can sometimes be explained. In 1970, Manin \cite{MAN70} gave a pairing between the \emph{Brauer group} $\Br X := H^2_{\text{\'et}}\left(X,\mathbb{G}_m\right)$ and the adelic points $X\left(\mathbb{A}_K\right)$ (equal to $\prod_{v \in \Omega_K}X(K_v)$ for $X$ projective). The kernel $X\left(\mathbb{A}_K\right)^{\Br}$ is closed in $X\left(\mathbb{A}_K\right)$ and contains $X\left(K\right)$. Hence, if $X\left(\mathbb{A}_K\right)^{\Br} \subsetneq X\left(\mathbb{A}_K\right)$, then $X$ fails weak approximation, and we say that there is a \emph{Brauer--Manin obstruction to weak approximation for $X$}.

While certain del Pezzo surfaces exhibit this obstruction (see \cite[\S9.4]{POO} for examples), Colliot-Th\'el\`ene has conjectured \cite[Conj.~14.1.2]{CTSk} that it is the only one to weak approximation for a large class of varieties including Fano varieties.

\begingroup
\setcounter{tmp}{\value{mydef}}
\setcounter{mydef}{0}
\renewcommand\themydef{\Alph{mydef}}
\begin{conjecture}[Colliot-Th\'el\`ene] \label{Conj:A}
For $X$ a geometrically rationally connected smooth projective variety over a number field $K$, we have $\overline{X\left(K\right)} = X\left(\mathbb{A}_K\right)^{\Br}$.
\end{conjecture}
\endgroup

For such $X$, we have $X\left(\mathbb{A}_K\right)^{\Br} = U \times \prod_{v \not\in S}X\left(K_v\right)$ for some finite set $S\subset\Omega_K$ and non-empty 
 open subset $U \subset \prod_{v \in S} X(K_v)$ \cite[(6),~p.~347]{CTSk}. Then Colliot-Th\'el\`ene's conjecture implies that any del Pezzo surface $X/K$ satisfies \emph{weak weak approximation}, meaning that $X\left(K\right)$ is dense in $\prod_{v \not\in S}X\left(K_v\right)$ for some finite $S$. 

Colliot-Th\'el\`ene and Sansuc \cite[p.~190]{CTS} raised the following question concerning the abundance of rational points on unirational varieties over a larger class of fields.

\begingroup
\setcounter{tmp}{\value{mydef}}
\setcounter{mydef}{1}
\renewcommand\themydef{\Alph{mydef}}
\begin{question}[Colliot-Th\'el\`ene--Sansuc] \label{Que:B}
For $X$ a projective unirational variety over a Hilbertian field $k$, is $X(k)$ always non-thin? That is, does $X$ have the Hilbert property?
\end{question}
\setcounter{mydef}{0}
\endgroup

See Definition \ref{def:hp} for the definitions of Hilbertianity and thin sets, and \cite[Question-Conjecture 1]{CZ} for a recently posed more extensive question when $k$ is a number field. Note that, by work of Colliot-Th\'el\`ene and Ekedahl \cite[Thm.~3.5.7]{SER}, weak weak approximation (over a number field) implies the Hilbert property.

There are no examples of del Pezzo surfaces with a rational point which are known to be non-unirational \cite[Rem.~9.4.11]{POO}. Thus, given a del Pezzo surface $X$ over a field $k$ with $X(k) \neq \emptyset$, proving that $X$ satisfies the Hilbert property when $k$ is Hilbertian and weak weak approximation when $k$ is a number field supports Conjecture \ref{Conj:A} and a positive answer to Question \ref{Que:B}, and is the focus of this paper.

Del Pezzo surfaces are classified by their \textsl{degree} $d:=K_X^2 \in \{1,\dots,9\}$, where $K_X$ is the canonical divisor. As dimension notionally governs arithmetic complexity for all Fano varieties, so does the degree for del Pezzo surfaces. Weak approximation when $X$ has a rational point holds for $d \geq 5$ \cite[Thm 9.4.8(vii)]{POO}, and counterexamples are given for all $d\leq4$ (\cite[\S9.4]{POO}). However, Salberger and Skorobogatov proved \cite[Thm.~0.1]{SS} that Conjecture \ref{Conj:A} holds for $d=4$ when $X$ has a rational point, thus implying weak weak approximation. Work of Swinnerton-Dyer \cite{SD} on cubic surfaces gives weak weak approximation when $d=3$, again under the assumption of a rational point. Examples of degree-$d$ del Pezzo surfaces with a Brauer--Manin obstruction to weak approximation have been given by Kresch and Tschinkel \cite[\S7]{KT} when $d=2$ and V\'arilly-Alvarado \cite[Thm.~1.1]{VAR} when $d=1$. However, little is known about weak weak approximation for $d=1,2$ outside of recent work of the first two authors \cite{DS}, which relies on the surface possessing a conic fibration. By adapting the strategies and methods of Swinnerton-Dyer used to prove unirationality, we shall prove weak weak approximation in the case $d=2$ without the requirement of a conic fibration.

\subsection{Results}
\begin{theorem} \label{thm:main1}
Let $X$ be a del Pezzo surface of degree $2$ over a number field $K$. Assume that there exists a point $P_0 \in X\left(K\right)$ lying neither on any $4$ of the $56$ exceptional curves of $X$, nor on the ramification divisor of the anticanonical morphism $\kappa\colon  X \rightarrow \mathbb{P}^2$. Then $X$ satisfies weak weak approximation.
\end{theorem}

\begin{remark}
Our method does not readily give an explicit description of the ``bad set'' of places $S$ implicit in the definition of weak weak approximation. In contrast, Swinnerton-Dyer (Theorem~\ref{thm:SD}) proved that for a (smooth) cubic surface $W \subset \mathbb{P}^3$, the set $S$ can be taken to be the set of (real archimedean places and) places $v$ such that the scheme-theoretic closure of $W$ in $\mathbb{P}^3_{\mathcal{O}_K}$ is singular over the prime ideal corresponding to $v$. It would be interesting to determine whether a similar bound on $S$ could be obtained in the degree-$2$ case, as would be implied by Conjecture \ref{Conj:A}.
\end{remark}

Our approach is similar to that of \cite{DS}, in that we iteratively produce covers of $X$ with desirable properties (including rationality of the source and geometric integrality of the generic fiber). Since the construction of these covers works over any field of characteristic zero (see Note \ref{note:poschar}), we obtain as a byproduct the following result (see Definition \ref{def:hu} for the definition of Hilbert-unirationality).

\begin{theorem} \label{thm:main2}
Let $X$ be a del Pezzo surface of degree $2$ over a field $k$ of characteristic zero. Assume that there exists a point $P_0 \in X\left(k\right)$ as in Theorem \ref{thm:main1}. Then $X$ is Hilbert-unirational. In particular, if $k$ is Hilbertian, then $X$ satisfies the Hilbert property.
\end{theorem}

We immediately obtain the following corollary for del Pezzo surfaces of degree $1$.

\begin{corollary}\label{cor}
Let $X$ be a del Pezzo surface of degree $1$ over a field $k$ of characteristic zero.
Assume that one of the $240$ exceptional curves of $X$ is defined over $k$.
\begin{enumerate}[label={\em (\roman*)}]
\item If $k$ is a number field, then $X$ satisfies weak weak approximation.
\item If $k$ is Hilbertian, then $X$ satisfies the Hilbert property.
\end{enumerate}
\end{corollary}

Indeed, blowing down the exceptional curve in Corollary \ref{cor}, we obtain a point $P_0$ on the resulting degree-$2$ del Pezzo surface satisfying the hypotheses of Theorem \ref{thm:main1}.

\subsection{Conventions and definitions}
We denote by $\Spec R$ the spectrum of a ring $R$ with the Zariski topology. A scheme over $R$ is a scheme $X$ equipped with a morphism $X \rightarrow \Spec R$. We denote by $X\left(R\right)$ the set of $R$-rational points of $X$, i.e.\ the set of sections of $X \rightarrow \Spec R$. Given morphisms of schemes $X \rightarrow Z$ and $Y \rightarrow Z$, we denote by $X \times_Z Y$ their fiber product. Given a scheme $X$ over $R$ and a morphism $\Spec S \rightarrow \Spec R$, we denote by $X_S$ the fiber product $X \times_{\Spec R} \Spec S$, the base change of $X$ to~$S$. A variety over a field $k$ is a geometrically integral separated scheme of finite type over $k$. We denote by $\overline{k}$ a separable closure of $k$. When working on a variety $X$ over $k$, we use the word ``points'' to refer to points in $X\left(\overline{k}\right)$ unless specified otherwise (and the notation $P \in X$ will accordingly mean that $P\in X(\overline{k})$), following classical terminology as in \cite[Ch.~1]{HAR}.

\begin{mydef}
For a divisor $D$ and points $P_1,\ldots,P_n$ on a variety $X$, we denote by $|D-\sum_{i=1}^na_iP_i|$ the linear system of divisors in $|D|$ containing each $P_i$ with multiplicity at least $a_i$.
\end{mydef}

\begin{mydef} \label{def:ratim}
Let $\phi\colon  Y \dashrightarrow X$ be a rational map of varieties. The \emph{domain} of $\phi$ is defined to be the the maximal open set $U \subset Y$ on which $\phi$ is defined as a morphism. The \emph{image} $\phi(V)$ of a subset $V \subset Y$ is defined to be $\phi(U\cap V)$. The \emph{closed image} of $\phi$ is defined to be the closure of $\phi\left(U\right)$ in $X$.
\end{mydef}

For a scheme $X$ of finite type over a number field $K$ and $S \subset \Omega_K$ finite, a {\em model} for $X$ over $\mathcal{O}_{K,S}$ is the datum of a scheme $\mathcal{X} \to \Spec \mathcal{O}_{K,S}$ of finite type and an isomorphism $\iota\colon X \xrightarrow{\sim}\mathcal{X}_K = \mathcal{X} \times_{\mathcal{O}_{K,S}} K$. We will always (abusively) identify $X$ with $\mathcal{X}_K$ via $\iota$.

\subsection{Acknowledgements}
We thank Jean-Louis Colliot-Th\'el\`ene, Daniel Loughran, \linebreak Cec\'ilia Salgado and Alexei Skorobogatov for useful discussions and feedback. We thank the anonymous referee for useful comments that improved the quality of the paper. Meetings to complete this work were made possible by funding from ICMS Edinburgh and Pierre Le Boudec's SNSF Professorship grant. The first named author started this project while being a guest at the Max Planck Institute of Bonn, which he thanks for its wonderful hospitality and optimal working conditions, and continued while being supported by Pierre Le Boudec's SNSF Professorship grant. The second named author was supported by the University of Bristol and the Heilbronn Institute for Mathematical Research. The third named author was supported by UKRI Fellowship MR/T041609/2. 
\section{Background}
\subsection{Hilbert property}\label{subsec:hilbert property}

\begin{mydef}\label{def:thin}
Let $X$ be a quasi-projective variety over a field $k$, and take $A \subset X\left(k\right)$.

We say that $A$ is of \emph{type I} if $A \subset W\left(k\right)$ for some closed $W \subsetneq X$, i.e.\ $A$ is not dense.

We say that $A$ is of \emph{type II} if $A \subset \phi\left(Y\left(k\right)\right)$ for $Y$ a variety and $\phi \colon Y \rightarrow X$ a generically finite separable dominant morphism of degree $\geq 2$.

We say that $A$ is \emph{thin} if it is contained in a finite union of type I and II subsets.
\end{mydef}

\begin{mydef} \label{def:hp}
Let $X$ be as in the previous definition. We say that $X$ has the \emph{Hilbert property} (over $k$) if $X\left(k\right)$ is not thin.

We say that a field $k$ is \emph{Hilbertian} if there exists a variety $X/k$ with the Hilbert property.
\end{mydef}

A field $k$ is Hilbertian if and only if $\mathbb{P}^1_k$ has the Hilbert property \cite[Prop.~13.5.3]{FJ}. Number fields are Hilbertian by Hilbert's irreducibility theorem, while local fields, finite fields and algebraically closed fields are not (see \cite[Ch.~13]{FJ}).

Given a dominant morphism of $k$-varieties $f\colon  Y \rightarrow X$ with geometrically integral generic fiber and a non-thin subset $A \subset Y\left(k\right)$, the image $f\left(A\right) \subset X\left(k\right)$ is not thin. Indeed, the pullback to $Y$ of a cover $\phi\colon  V \rightarrow X$ as in the definition of type II thin sets is also dominant and generically finite of degree at least $2$, hence one may deduce this fact by contradiction. (See also \cite[Prop.~7.13]{CTS}.) Since rational varieties over Hilbertian fields satisfy the Hilbert property (using \cite[Thm.~13.4.2]{FJ} and the fact that the Hilbert property is a birational invariant), rationality of $Y$ implies that $X$ satisfies the Hilbert property. This leads to the following definition, put forward by Colliot-Th\'el\`ene in a talk given for the Algebra, Geometry and Physics seminar of HU Berlin/MPIM Bonn in 2021.

\begin{mydef}\label{def:hu}
We say that a variety $X$ over a field $k$ is \emph{Hilbert-unirational} if there exists a rational variety $Y$ over $k$ and a dominant morphism $f\colon  Y \rightarrow X$ such that $f$ has geometrically integral generic fiber.
\end{mydef}

From the above, it follows immediately that if $X$ is Hilbert-unirational and $k$ is Hilbertian, then $X$ has the Hilbert property.

\subsection{Weak weak approximation}

We recall the following definition.
\begin{mydef}\label{def:WWA}
We say that a smooth variety $X$ over a number field $K$ satisfies \emph{weak weak approximation} if $X(K)$ is dense in $\prod_{v \notin S} X\left(K_v\right)$ for some finite set $S$ of places of $K$.
\end{mydef}

\begin{mydef}We say that a scheme $Y$ of finite type over a perfect field is
\textsl{split} if there exists a non-empty geometrically integral open subscheme $U\subset Y$.
\end{mydef}

Analogously to the Hilbert property, the existence of rational covers of $X$ with certain properties guarantees weak weak approximation.

\begin{mydef}\label{def:su}
We say that a smooth variety $X$ over a field $k$ is \emph{split-unirational} if there exists a rational variety $Y$ over $k$ and a dominant morphism $f\colon  Y \rightarrow X$ such that:
\begin{itemize}
    \item[$(a)$] $f$ has geometrically integral generic fiber; and,
    \item[$(b)$] every birational modification $f'\colon Y' \to X'$  of $f$ with $Y'$ and $X'$ smooth has split fibers over all codimension-$1$ points.
\end{itemize}
\end{mydef}

The second condition is equivalent to $f$ having ``split reduction'' at every divisorial discrete valuation ring $R \in \operatorname{DVR}(k(X),k)$ (see \cite[\S4]{DS}). Denef \cite{DEN} proved that every $f$ satisfying $(a)$ and $(b)$ above, when $k$ is a number field $K$, is \emph{arithmetically surjective}.

\begin{mydef}\label{def:ar.surj.}
We say that a morphism $f\colon Y \rightarrow X$ of varieties over a number field $K$ is \emph{arithmetically surjective} if there exists a finite set of places $S$ such that we have $f(Y(K_v))=X(K_v)$ for all $v \notin S$.
\end{mydef}

By Denef's result, every split-unirational variety $X$ over a number field $K$ satisfies weak weak approximation: let $(P_v) \in \prod_{v \notin S} X(K_v)$ and $(Q_v) \in \prod_{v \notin S} Y(K_v)$ be such that $f(Q_v)=P_v$; approximating $(Q_v)$ by $Q \in Y(K)$, the point $P=f(Q)$ approximates $(P_v)$.  

We will use the following result, giving a sufficient condition for split-unirationality.

\begin{proposition}\label{prop:ssfcover}
    Let $X$ be a smooth proper variety defined over a number field $K$, and let $f\colon Y \to X$ be a proper morphism with $Y$ rational such that there exists an open subscheme $U \subset Y$ such that the restriction $f|_U\colon U \to X$ is smooth and has split fibers. Then $f$ is split-unirational, hence arithmetically surjective.
\end{proposition}

\begin{proof}
This follows directly from \cite[Prop.~4.4]{DS}.
\end{proof}

\begin{remark}
    The specific instance of split-unirationality in Proposition \ref{prop:ssfcover} is supported by a folklore argument (not invoking Denef's result) combining the Lang--Weil--Nisnevich estimates and Hensel's lemma. A sketch can be found in \cite[pp.\ 42-43]{CTSSK}. We present a detailed proof here for completeness. In the proof of the lemma we use the terminology of schemes instead of varieties (and so points refer to scheme-theoretic points). 
\end{remark}

\begin{lemma}\label{lem:arithmetically surjective}
    Let $K$ be a number field and $f\colon V \to Y$ a smooth morphism with split fibers onto a smooth proper $K$-variety $Y$. Then $f$ is arithmetically surjective.
\end{lemma}
\begin{proof}
We first prove that, given $\mathcal{Y}$ a reduced $\mathcal{O}_{K,S}$-scheme of finite type and $\phi\colon \mathcal{V} \to \mathcal{Y}$ a morphism of finite type with split fibers (we require neither smoothness of $\phi$ nor properness of $\mathcal{Y}$), then, after possibly enlarging the finite set $S$, the morphism $\phi\colon \mathcal{V}(\F_v) \to \mathcal{Y}(\F_v)$ is surjective for all $v \notin S$. We use Noetherian induction on $\mathcal{Y}$. 
    
Let $\eta \in \mathcal{Y}$ be a generic point, and $V_1$ be a non-empty geometrically integral open subscheme of the generic fiber $\phi^{-1}(\eta)$. We may assume that $V_1$ is affine, thus quasi-projective, i.e.\ that there exists an $\kappa(\eta)$-isomorphism $V_1 \cong Z_1 \setminus Z_2 \subset \P^n_{\kappa (\eta)}$, with $Z_1,Z_2$ closed in $\P^n_{\kappa(\eta)}$ (where $\kappa(\eta)$ is the residue field of $\eta$). We may take $Z_1$ to be the Zariski-closure of $V_1$, and thus assume that it is geometrically integral and that $\dim Z_2 < \dim Z_1$.  We may then spread out the following statement:
    \begin{center}
        there exists a $\kappa (\eta)$-isomorphism $V_1 \cong Z_1\setminus Z_2 \subset \P^n_{\kappa (\eta)}$, with $Z_1,Z_2$ ($\dim Z_2 < \dim Z_1$) closed in $\P^n_{\kappa (\eta)}$ and flat over $\Spec \kappa (\eta)$, and such that the fibre $Z_1 \to \Spec \kappa (\eta)$ is geometrically integral
    \end{center}
(the flatness over $\Spec \kappa (\eta)$ is automatic as this is the spectrum of a field) to obtain an open $\mathcal{U} \subset \mathcal{Y}$ (containing $\eta=\Spec \kappa (\eta)$), and an open subscheme $\mathcal{V}_1\subset \phi^{-1}(\mathcal{U})$ such that:
    \begin{center}
        there exists a $\mathcal{U}$-isomorphism $\mathcal{V}_1 \cong \mathcal{Z}_1\setminus \mathcal{Z}_2 \subset \P^n_{\mathcal{U}}$,
        with $\mathcal{Z}_1,\mathcal{Z}_2$ ($\dim \mathcal{Z}_2< \dim \mathcal{Z}_1$) closed in $\P^n_{\mathcal{U}}$ and flat over $\Spec \mathcal{U}$, and such that $\mathcal{Z}_1 \to \Spec \mathcal{U}$ has geometrically integral fibers.
    \end{center}

We refer to \cite[\S3.2]{POO} (and Table 1 at pp.\ 306-307 of {\em loc.\ cit.}) for the proof that all the required properties actually spread out, thus showing the existence of such a $\mathcal{U}$ and $\mathcal{V}_1$. We set $n_i:= \dim Z_i= \dim \mathcal{Z}_i- \dim \mathcal{U}, i=1,2.$ Note that $n_2<n_1$.

For any finite field $\mathbb{F}_q$, the Lang--Weil--Nisnevich theorem gives an asymptotic estimate, for every geometrically integral closed subscheme $Z \subset \P^n_{\F_q}$ \cite[Thm.~1]{LW}(resp.\ for every reduced closed subscheme $Z \subset \P^n_{\F_q}$ \cite[Lem.~1]{LW}) of the form
    \[
    \# Z(\F_q)= q^{\dim Z} + O_{\dim Z,n,d}(q^{\dim Z- \frac{1}{2}})
    \]
(resp.\ $\# Z(\F_q)=O_{\dim Z,n,d}(q^{\dim Z})$), where the implicit constant in the ``big $O$'' depends only on $\dim Z, n$ and the degree $d$ of $Z$. Since the Hilbert polynomial (from which dimension and degree may be inferred) remains constant in flat projective families \cite[Thm.~III.9.9]{HAR} (in particular, the families $\mathcal{Z}_1 \to \mathcal{U}$ and $\mathcal{Z}_2 \to \mathcal{U}$), these estimates show that there exists some constant $N_0 >0$ such that for all non-archimedean $v$ with norm $q_v > N_0$ (which we may assume to hold for all $v \notin S$ after enlarging $S$), the $n_1$-dimensional fibers of $\mathcal{Z}_1 \to \mathcal{U}$ have more $\F_v$-points than the $n_2$-dimensional ($n_2 <n_1$) fibers of $\mathcal{Z}_2 \to \mathcal{U}$. Since $\mathcal{V}_1 \cong \mathcal{Z}_1 \setminus \mathcal{Z}_2$ as $\mathcal{U}$-schemes, we thus obtain that all fibers of $\mathcal{V}_1(\F_v) \to \mathcal{U}(\F_v)$ are non-empty, hence that the image of $\mathcal{V}(\F_v) \to \mathcal{Y}(\F_v)$ contains $\mathcal{U}(\F_v)$. 
    
Applying the Noetherian induction hypothesis on $\phi$ restricted over the complement $\mathcal{Y}'$ of $\mathcal{U}$ in $\mathcal{Y}$, we may assume (after enlarging $S$ if necessary) that the image of $\mathcal{V}(\F_v) \to \mathcal{Y}(\F_v)$ contains $\mathcal{Y}'(\F_v)=\mathcal{Y}(\F_v) \setminus \mathcal{U}(\F_v)$. Combined with the above, this proves surjectivity.
    
We now show how to deduce the lemma. Let $S \subset \Omega_K$ be finite and such that there exist $\mathcal{O}_{K,S}$-models $\mathcal{V}$ and $\mathcal{Y}$ of $V$ and $Y$ and $f$ extends to a morphism $\phi\colon \mathcal{V} \to \mathcal{Y}$. Spreading out the property of having split fibers (that this spreads out follows from spreading out geometric integrality of fibers \cite[Thm.~9.7.7]{EGA} after restriction to an open intersecting the generic fiber in a geometrically integral subscheme) and using the just proven statement, we may assume (after possibly enlarging $S$) that $\phi(\mathcal{V}(\F_v)) = \mathcal{Y}(\F_v)$ for all $v \notin S$.

After enlarging $S$ again, we may also assume by spreading out that $\phi$ is smooth and $\mathcal{Y}$ is proper over $\mathcal{O}_{K,S}$. Let $v \in \Omega_K \setminus S$ and $P_v$ be a point in $Y(K_v)$. Since $\mathcal{Y}$ is proper, $P_v$ extends to an $\mathcal{O}_v$-point of $\mathcal{Y}$. The restriction of $\phi$ over this $\mathcal{O}_v$-point gives a smooth model $\mathcal{Y}_{P_v} \to \Spec \mathcal{O}_v$ of the $K_v$-variety $\phi^{-1}(P_v):=Y_{P_v}$. By the surjectivity above, we know that this smooth $\mathcal{O}_v$-model has an $\F_v$-point $\Tilde{Q}$ on its special fiber. Choosing a standard smooth neighborhood (see \cite[Tags~01V5,~01V7]{SP}) $\Spec \mathcal{O}_v[x_1,\ldots,x_N]/(f_1,\ldots,f_M), \ M \leq N$ of $\Tilde{Q}$ in the smooth $\mathcal{O}_v$-scheme $\mathcal{Y}_{P_v}$, we may use Hensel's lemma to lift the smooth $\F_v$-point $\Tilde{Q}:=(x_1(\Tilde{Q}),\ldots,x_N(\Tilde{Q}))$ to an $\mathcal{O}_v$-point of $\Spec \mathcal{O}_v[x_1,\ldots,x_N]/(f_1,\ldots,f_M) \subset \mathcal{Y}_{P_v}$, which then induces a $K_v$-point on the generic fiber $Y_{P_v}$ of the model $\mathcal{Y}_{P_v}$. Thus ${Y}_{P_v}(K_v)\neq \emptyset$ and $V(K_v) \to Y(K_v)$ is surjective ($P_v$ was arbitrary), as wished.
\end{proof}

We summarize the previous two subsections with the following proposition.
\begin{proposition}\label{prop:goodcover}
A Hilbert-unirational (resp.\ split-unirational) variety $X$ over a Hilbertian field (resp.\ a  number field) satisfies the Hilbert property (resp.\ weak weak approximation).
\end{proposition} 

We will prove Theorem \ref{thm:main1} by constructing a split-unirational cover of the del Pezzo surface $X$. Since a split-unirational cover is a Hilbert-unirational cover, our construction also gives Theorem \ref{thm:main2}. However, a Hilbert-unirational cover already appears as an intermediate stage in our construction, hence we will prove Theorem \ref{thm:main2} first.

\subsection{Del Pezzo surfaces}

\begin{mydef}
Let $X$ be a smooth projective geometrically integral variety over a field~$k$. We say that $X$ is a \emph{Fano variety} if the anticanonical class $-K_X$ is ample. A \emph{del Pezzo surface} is a Fano variety of dimension two.
\end{mydef}

For the rest of this section, let $X$ be a del Pezzo surface over a field $k$. The surface $X_{\overline{k}}$ is isomorphic either to $\mathbb{P}^1 \times \mathbb{P}^1$ or to the blowup of $\mathbb{P}^2$ of at most eight points in general position (see \cite[\S9.4]{POO}).

\begin{mydef}
The \emph{degree} of $X$ is the integer $d = d\left(X\right) := \left(K_X\right) \cdot \left(K_X\right)$, where $\cdot$ denotes the intersection product on $\Pic \left(X\right)$.
\end{mydef}

It can be easily shown that we have $d\left(X\right) \in \{1,\dots,9\}$ \cite[Lem.~24.3.1]{MAN}.

\subsubsection{Anticanonical morphism}
The linear system $|{-K_X}|$ is base-point-free if and only if $d \neq 1$ \cite[Prop.~III.3.4.1]{KOL}. Since base-point-free linear systems induce morphisms to projective space \cite[\S II.7]{HAR}, we may consider for any del Pezzo surface $X$ of degree $d \geq 2$ the anticanonical morphism $\kappa\colon  X \rightarrow \mathbb{P}^d$ (it follows from \cite[Cor.~III.3.2.5.2]{KOL} that the dimension of $|{-K_X}|$ is $d = K_X^2$).

For $d \geq 3$ the class $-K_X$ is very ample, hence $\kappa$ is an embedding. In particular, when $d=3$, the image of $\kappa$ is a smooth cubic surface in $\mathbb{P}^3$.

When $d=1$, the linear system $|{-K_X}|$ consists of curves of arithmetic genus $1$ and has a unique, $k$-rational base point, thus inducing a rational map to $\mathbb{P}^1$ away from the base point. Blowing up the base point gives a geometrically rational elliptic surface over $\mathbb{P}^1$.

When $d=2$, the morphism $\kappa\colon  X \rightarrow \mathbb{P}^2$ is finite of degree $2$ and ramifies over a quartic curve $\mathcal{B}$. We denote by $\mathcal{R} \subset X$ the ramification divisor of $\kappa$, which we call the \textsl{ramification divisor of $X$}. The surface $X$ is of the form
$$
X:  w^2 +fw= g \subset \mathbb{P}\left(1,1,1,2\right)
 $$
 in the weighted projective space $\mathbb{P}\left(1,1,1,2\right)$ with variables $x,y,z,w$, where $f,g\in k[x,y,z]$ are homogeneous polynomials of degrees 2 and 4, respectively. The curve $\mathcal{B}\subset\mathbb{P}^2$ is then defined by $f^2+4g=0$, and it is smooth when char $k\neq2$.
 The morphism $\kappa$ induces the \emph{Geiser involution} $\iota\colon  X \rightarrow X$, which swaps the two sheets of the double cover.
\noindent The embedding $\kappa_2\colon  X \hookrightarrow \mathbb{P}^6$ induced by $-2K_X$ realizes $X$ as an octic (degree-$8$) surface contained in the cone over the Veronese surface in the hyperplane $x_6 = 0$ (identified with $\mathbb{P}^5$) with vertex $[0:\ldots:0:1]$. Denoting the 2-uple embedding by $\rho_2$ and the aforementioned hyperplane projection by $\pi_6$, we have the following commutative diagram:
 
\begin{center} 
\begin{tikzcd}
X \arrow[r,hook,"\kappa_2"] \arrow[d,"\kappa"] & \mathbb{P}^6 \arrow[d,dotted,"\pi_6"] \\
\mathbb{P}^2 \arrow[r,hook,"\rho_2"] & \mathbb{P}^5 \\
\end{tikzcd}
\end{center}
 \subsubsection{Lines}
\begin{mydef}
    Let $X$ be a del Pezzo surface. An irreducible curve $E \subset X$ is called an \emph{exceptional curve} (or \emph{line}) on $X$ if $E \cdot E = -1$ (equiv.\ $ E \cdot K_X = -1$). 
\end{mydef}

The terminology ``line'' comes from the study of exceptional curves on the surfaces studied by del Pezzo himself, namely those of degree $d \geq 3$, for which the exceptional curves on $X$ are precisely the lines on $X$ under its anticanonical embedding.

When $d=2$, there are $56$ exceptional curves coming in $28$ Geiser-involute pairs. Following \cite{STVA}, we define a \textsl{bitangent line to $\mathcal{B}$} to be a line $L$ such that $\kappa^{-1}(L)$ is geometrically reducible (when char $k\neq2$, these are the lines that have even intersection multiplicity everywhere with $\mathcal{B}$). Each of the 28 pairs of exceptional curves is the inverse image under $\kappa$ of a bitangent line to $\mathcal{B}$ \cite[Lem.~2.1]{STVA}. (See \cite[\S6.1.1]{DOL} for more on the existence and configuration of the bitangents of a smooth plane quartic.) For these surfaces the maximal number of concurrent exceptional curves is $4$, motivating the following definition.
 
 \begin{mydef}
 A point $P \in X\left(\overline{k}\right)$ on a degree-$2$ del Pezzo surface $X$ over a field $k$ is a \emph{generalized Eckardt point} if it lies on $4$ of the $56$ exceptional curves of $X$.
 \end{mydef}
 
 This etymology is rooted in the study of smooth projective cubic surfaces (del Pezzo surfaces of degree $3$). It is a famous result (the \emph{Cayley--Salmon theorem}) that a smooth projective cubic surface $W \subset \mathbb{P}^3$ (over an algebraically closed field) contains exactly $27$ lines. Any point on a smooth cubic surface lies on at most $3$ of the $27$ lines, and points lying on $3$ of the lines are known as \emph{Eckardt points}.

\begin{remark}
As with Eckardt points, generalized Eckardt points are rare: on a general del Pezzo surface of degree $2$ there are no generalized Eckardt points (see \cite[Ex.~7]{STVA}).
\end{remark}
\section{Two geometric procedures}\label{sec:twogeomproc}
In this section, we describe two geometric constructions which we will use to build covers of a degree-$2$ del Pezzo surface with good geometric properties. 

We begin with analogous procedures for cubic surfaces (which we informally name the ``point procedure'' and ``rational curve procedure'') and their relationship to arithmetic.

\subsection{Cubic surfaces}
Let $W$ be a smooth cubic surface over an infinite field $k$ in $\mathbb{P}^3$ with $W\left(k\right) \neq \emptyset$. We illustrate two geometric procedures employed by Segre \cite{SEG} and Manin \cite{MAN} to prove that $W$ is unirational, and explain how these procedures were employed by Swinnerton-Dyer \cite{SD} to demonstrate that $W$ is Hilbert-unirational. First, we give a map which propagates points (the ``point procedure'').
\begin{mydef}\label{def:dp3pts}
Define $\phi_W\colon  W \times W \dashrightarrow W$ to be the map sending two distinct points $P,Q \in W$ to the third point of the intersection $L_{P,Q} \cap W \subset \mathbb{P}^3$ (where defined), where $L_{P,Q}$ is the line in $\mathbb{P}^3$ through $P$ and $Q$.
\end{mydef}
Here the intersection $L_{P,Q} \cap W$ is understood scheme-theoretically, so that it may include a non-reduced point. Provided that $L_{P,Q}$ is not one of the $27$ lines of $W$, the intersection $L_{P,Q} \cap W$ does indeed consist of three points with multiplicity.
\begin{remark}
Note that $\phi_W$ operates along anticanonical curves of $W$ in the following sense: given general $P,Q \in W$, there exists a pencil $|{-K_W - P -Q}|$ of plane sections of $W$ containing $P$ and $Q$, namely those cut out by planes containing the line $L_{P,Q}$. The point $\phi_W\left(P,Q\right)$ is the third point at which all of these curves meet, i.e.\ the third and final base-point of $|{-K_W - P -Q}|$.
\end{remark}

We now describe the ``rational curve procedure'' that produces a rational curve from a rational point.
\begin{mydef}\label{def:dp3curves}
Let $P \in W(k)$ be a rational point not on any line. Define $C_P$ to be the intersection $T_PW \cap W$, where $T_P W$ is the tangent plane to $W$ at $P$. Then $C_P$ is an irreducible singular cubic curve in the plane $T_P W$, hence rational. Define $\psi_P\colon  \mathbb{P}^1 \rightarrow W$ to be an associated non-constant birational morphism (coming from desingularising $C_P$).
\end{mydef}
\begin{remark}
The curve $C_P \subset W$ admits an alternative definition, as observed by Manin \cite[Proof~of~Thm.~29.4]{MAN}: provided that $P$ does not lie on any of the $27$ lines of $W$, we may blow up $W$ at $P$ to obtain a del Pezzo surface $X$ of degree $2$. Letting $E \subset X$ denote the exceptional divisor of this blowup, $E$ is one of the $56$ lines of $X$, which come in $28$ Geiser-conjugate pairs $\left(L,-K_X-L\right)$. The curve $C_P$ is the image on $W$ of the Geiser conjugate $-K_X - E$ of $E$.

\end{remark}
\subsubsection{Unirationality}
Combining the two procedures above leads to the following result.

\begin{theorem}\label{thm:unir} \cite[Thm.~29.4,~Thm.~30.1]{MAN}
Let $W$ be a smooth cubic surface over a field $k$ which has at least $37$ elements, and suppose that $W(k) \neq \emptyset$. Then $W$ is unirational.
\end{theorem}

Koll\'ar \cite{KOL} later extended this result to include smaller finite fields, thus settling unirationality for all cubic hypersurfaces with a rational point.
The proof of Theorem~\ref{thm:unir} supplied by Manin in \cite{MAN}, originating from ideas of Segre \cite{SEG}, comes in two parts: proving unirationality given a rational point avoiding all $27$ lines (which is done over any field, not necessarily infinite), and proving that one rational point is enough to deduce the existence of a point lying on none of the lines (which is done over large enough finite fields and infinite fields). For the first part, one simply iterates the rational curve procedure in Definition \ref{def:dp3curves} twice, once through the initial rational point and once through its generic point. For the second part, one employs the map $\phi_W$ in Definition \ref{def:dp3pts} in order to produce two points conjugate over a quadratic extension and from them two curves, picking up infinitely many points at the cost of moving to a higher-degree field. By then employing the map $\phi_W$ on conjugate pairs, one is able to find one pair for which the output is a rational point over the original ground field. For an explanation of these and other techniques used in the study of unirationality for cubic surfaces and their higher-dimensional analogues, the reader is encouraged to consult the aforementioned work of Koll\'ar \cite{KOL}, who employs and extends the techniques of Segre--Manin.

\subsubsection{Hilbert-unirationality}
With the same tools, Swinnerton-Dyer was able to prove that cubic surfaces satisfying the hypotheses of Theorem \ref{thm:unir} are Hilbert-unirational. 
\begin{proposition}[{\cite[Lem.~7]{SD}}]\label{Prop: SD}
Let $W$ be a smooth cubic surface over a field $k$, let $P_0 \in W\left(k\right)$ be a rational point not on any line of $W$, and denote by $\psi$ the morphism $\psi_{P_0}\colon  \mathbb{P}^1 \rightarrow W$ as defined above. Define $Z \subset W \times \mathbb{P}^1 \times \mathbb{P}^1 \times W$ to be the set
$$
\left\{\left(P_1,Q_2,Q_3,P_4\right) \in W \times \mathbb{P}^1 \times \mathbb{P}^1 \times W   : \;\begin{gathered}
    P_1 \neq \psi(Q_2),\  \psi(Q_3) \text{ does not lie on any line, } \\
    \phi_W\left(\psi\left(Q_2\right),P_4\right) = P_1, \text{ and }P_4 \in C_{\psi\left(Q_3\right)}.\end{gathered} \right\}.
$$
Let $\pi_1\colon  Z \rightarrow W$ denote the restriction of projection to the first factor of $W \times \mathbb{P}^1 \times \mathbb{P}^1 \times W$. Then $Z$ is a rational variety and the generic fiber of $\pi_1$ is geometrically integral. That is, $W$ is Hilbert-unirational.
\end{proposition}
We immediately obtain the following corollary.

\begin{corollary}
Smooth cubic surfaces over Hilbertian fields with a rational point satisfy the Hilbert property. 
\end{corollary}

Key to the proof of Proposition \ref{Prop: SD} is an understanding of certain curves on $W$ constructed using the maps $\phi_W$ and $\psi_{P}$. One such curve is the image of the map $\mathbb{P}^1 \rightarrow W$, $Q_2 \mapsto \phi_W\left(\psi\left(Q_2\right),P_1\right)$ for fixed $P_1 \in W$. Another is the locus of points $P' \in W$ with $P' \in C_P$ for given $P \in W$. Specifically, Swinnerton-Dyer considers the projection of a general fiber of $\pi_1$ onto $\mathbb{P}^1 \times \mathbb{P}^1$, and it is by determining the properties of such projections (including their divisor class and singularities) that he ultimately deduces the geometric integrality of this birational image, hence of general fibers.

As mentioned in the introduction, Swinnerton-Dyer then pushes forward his method to prove the following result on weak weak approximation.

\begin{theorem}[Swinnerton-Dyer] \cite[Thm.~1]{SD} \label{thm:SD}
Let $W \subset \mathbb{P}^3_K$ be a smooth cubic surface over a number field $K$. Assume that $W(K) \neq \emptyset$, and let $S \subset \Omega_K$ be the union of all real places and all finite places corresponding to prime ideals $\mathfrak{p}$ such that the scheme-theoretic closure of $W$ in $\P^3_{\mathcal{O}_K}$ has  a singular fiber over $\mathfrak{p}$. Then $W(K)$ is dense $\prod_{v \notin S} W\left(K_v\right)$.
\end{theorem}

\subsection{Del Pezzo surfaces of degree $2$}
We now present analogues of the rational curve and point procedures for cubic surfaces in the setting of del Pezzo surfaces of degree $2$. The rational curve procedure, first given by Manin \cite{MAN} in certain cases and later generalized by Salgado, Testa and V\'arilly-Alvarado \cite{STVA}, is key to the associated proofs of unirationality. The point procedure we give is new, although it bears an interesting resemblance to an alternative formulation of the rational curve procedure observed by Festi and van Luijk \cite[Prop.~4.3]{FvL} in the extended preprint of \cite{FvLfull}.

Let $X$ be a del Pezzo surface of degree 2 over a field $k$. For convenience, we make the following definition.

\begin{mydef} \label{def:gen}
Let $P \in X$ be a point not lying on the ramification divisor $\mathcal{R}$ of $X$. We say that $P$ is \emph{general} if it is not a generalized Eckardt point, and we say that it is \emph{very general} if it does not lie on any exceptional curve of $X$.
\end{mydef}

The hypothesis of Theorem \ref{thm:main1} is that $X$ contains a general rational point. By the following result, this is equivalent to the existence of a very general rational point for $\ch k =0$. \textbf{Thus, we henceforth assume without loss of generality that $X$ contains a very general point $P_0 \in X(k)$.}

\begin{theorem} \cite[Cor.~3.3]{STVA} \label{thm:STVA}
Let $X$ be a del Pezzo surface of degree $2$ over a field $k$. If $X$ contains a general $k$-rational point, then $X$ is unirational.
\end{theorem}

\subsubsection{Rational curve procedure}\label{subsubsec: making rational curves}
A key ingredient in the proof of Theorem \ref{thm:STVA} is the following construction, given a general point $P \in X\left(k\right)$, of a rational curve containing $P$.

\begin{mydef}\label{def:dp2curves}
Let $P \in X(k)$ be a general point, and let $\rho\colon  X' = \text{Bl}_P X \rightarrow X$ be the associated blowup, with exceptional divisor $E\subset X'$ above $P$. Denote by $E_1,\dots,E_n$ the $n \leq 3$ exceptional curves of $X$ on which $P$ lies, and denote by $E_i' \subset X'$ the strict transform of $E_i$. The curve $C_P \subset X$ is defined to be the image via $\rho$ of the unique curve in the linear system $|{-2K_{X'} - E - \sum_{i=1}^n E_i'}|$,
which is then the unique curve in the system $|{-2K_X - \sum_{i=1}^n E_i - \left(3-n\right)P}|$.
\end{mydef}

The existence, uniqueness, integrality and rationality of $C_P$ is proved in \cite[Thm.~3.1]{STVA}.

\begin{note}\label{Note: multiplicity}
Note that, in the above definition, we have $(-2 K_{X^{\prime}}-E-\sum_{i=1}^n E_i^{\prime})\cdot E= 3 -n$. Thus, when $P$ lies on $n\leq2$ exceptional curves, the curve $C_P$ contains $P$ with multiplicity $3-n$. However, when $P$ lies on exactly three exceptional curves, the curve $C_P$ is a $k$-rational exceptional curve on $X$ not containing $P$. This leads to the following interesting observation: when $P$ lies on exactly three exceptional curves, one may blow down $C_P$ and obtain a del Pezzo surface of degree $3$ with a general rational point.
\end{note}

\begin{note}
When $P$ is very general, one can describe $C_P$ in a similar way to Definition~\ref{def:dp3curves}. Consider the octic embedding $\kappa_2\colon  X\hookrightarrow \P^6$ defined by $|-2K_X|$, and let $H:=\{h=0\}$ be the unique hyperplane that passes through $P$ and osculates $X$ at it, i.e. such that the restriction of $h$ to $X$ is $0$ to the second order (in Taylor expansion) at the point $P$. Then $C_P$ is defined as the intersection $H \cap X$. Indeed, such a hyperplane corresponds by construction to the unique curve in the linear system $|{-2K_X -3 P}|$.
\end{note}

\begin{note}
If $P$ lies on $4$ exceptional curves $E_1,\dots,E_4$, then $E_1 + E_2 + E_3 + E_4 \sim -2K_X$ (this is part of \cite[Prop.~4.12]{FvLfull}), thus the system $|{-2K_X - \sum_{i=1}^4 E_i}|$ contains no non-zero effective divisors. 
\end{note}

\subsubsection{The points procedure}\label{subsubsec: propagating points}
Having established an analogue of the Segre--Manin procedure for making rational curves, we now come to a procedure for producing one point from two others. We will make use of the fact that anticanonical curves on a degree-$2$ del Pezzo surface are of arithmetic genus $1$, as can be seen from the adjunction formula \cite[Exercise~V.1.3]{HAR}. Recall that the linear system $|{-K_X}|$ induces a degree-$2$ morphism $\kappa\colon  X \rightarrow \mathbb{P}^2$. Given two points $P, Q \in X$ such that $\kappa\left(P\right) \neq \kappa\left(Q\right)$, consider the preimage $E_{P,Q}$ in $X$ of the line $L = L_{\kappa\left(P\right),\kappa\left(Q\right)}\subset\mathbb{P}^2$ joining $\kappa\left(P\right)$ and $\kappa\left(Q\right)$. This curve $E_{P,Q}$ lies in the anticanonical system $|{-K_X}|$, hence it is of arithmetic genus $1$. If $\ch k \neq 2$, it is singular if and only if $L$ is tangent to the (smooth) plane quartic branch curve $\mathcal{B}$ of $\kappa$ and, for any $\ch k$, geometrically reducible if and only if $L$ is one of the $28$ bitangents of $\mathcal{B}$ \cite[Lem.~4]{STVA}, in which case it is the union of two exceptional curves.

\begin{mydef}
Define $\phi\colon  X \times X \dashrightarrow X, \left(P,Q\right) \mapsto R$, where $R$ is the unique point of $E_{P,Q}$ such that $2\left(\iota\left(P\right)\right) - \left(Q\right) \sim \left(R\right) \in \Div E_{P,Q}$. 
\end{mydef}

\begin{remark}
When $\ch k \neq 2$, the map $\phi$ is clearly defined on all $(P,Q)\in X\times X$ with $\kappa\left(P\right) \neq \kappa\left(Q\right)$ (meaning $Q \not\in \{P,\iota\left(P\right)\}$), such that $L_{\kappa\left(P\right),\kappa\left(Q\right)}$ is not tangent to the branch curve $\mathcal{B} \subset \mathbb{P}^2$ (i.e.\ such that $E_{P,Q}$ is smooth), see \cite[Prop.~III.3.4]{SIL}. We will show in Proposition \ref{prop:domain}  that the domain of $\phi$ actually includes (in any characteristic) also the pairs of points $(P,Q), \kappa(P) \neq \kappa(Q)$ for which $E_{P,Q}$ is irreducible and both $P$ and $Q$ are smooth points on it. Moreover, as we will show in Remark~\ref{rem:inv},  assuming $\ch k \neq 2$,  fixing $P\in X$, the map $\phi(P,-)\colon X\dashrightarrow X$ extends to a regular involution on $X\setminus\{P,\iota(P),C_P\}$ (where $C_P$ is only excluded if it is well-defined, i.e. when $P$ is a general point).  
\end{remark}

\begin{remark}\label{rem:phidef}
We make two observations on the above definition and its connection to the other geometric procedures that we have seen.
\begin{enumerate}
\item One reason for using $\iota\left(P\right)$ instead of $P$ in the above definition, as might appear more natural, is that the given map is connected to the rational curve procedure. Festi and van Luijk show \cite[Prop.~4.3]{FvL} that, writing $D_Q$, $Q \in \mathbb{P}^1$ for the pencil of curves in $|{-K_X}|$ through both $P$ and $\iota\left(P\right)$, the curve $C_P$ is formed by computing on each $D_Q$ the unique point $R$ such that $2\left(\iota\left(P\right)\right) - \left(P\right) \sim \left(R\right) \in \Div D_Q$. In particular, both the rational curve and point procedures arise from constructing degree-$1$ divisors from pairs of points on anticanonical curves.
\item The map $\phi$ admits a geometric interpretation akin to the line map employed by Swinnerton-Dyer: for Zariski-general $P$ and $Q$, the point $\phi\left(P,Q\right)$ is the third base-point of $|{-2K_X - 2P - Q}|$, just as the point $\phi_W\left(P,Q\right)$ is the third base-point of $|{-K_W - P - Q}|$.  In fact, note first that the linear system $|-2K_X - 2P - Q|$ contains the linear subsystem of reducible divisors consisting of the sum of $E_{P,Q}$ and any curve from $|-K_X - P|$. The base locus of this linear subsystem is (set-theoretically) the curve $E_{P,Q}$, hence any base-point lies on $E_{P,Q}$. For an element $H \in \left|-2 K_X-2 P-Q\right|$ (a hyperplane in $\mathbb{P}^6$, into which $X$ embeds via the morphism coming from $|-2K_X|$), writing $H \cap E_{P,Q} = \{P,P,Q,R_H\}$ as four points with multiplicity on $E_{P,Q}$, we have $2\left(P\right) + \left(Q\right) + \left(R_H\right) \sim (-2K_X)|_{E_{P,Q}} \sim 2 (-K_X)|_{E_{P,Q}} \sim 2\left(P\right) + 2\left(\iota\left(P\right)\right)$ on $E_{P,Q}$. Then $\left(R_H\right) \sim 2\left(\iota\left(P\right)\right) - \left(Q\right)$ is constant in $H$, and is the unique third base point of $|{-2K_X - 2P - Q}|$.
\end{enumerate}
\end{remark}

\begin{note}
It would be of great interest to know whether a tactic similar to that of Segre and Manin could be used in order to prove unirationality in the case where one begins with a non-general rational point, i.e.\ whether one can use the $\phi$ map to deduce the existence of a general rational point from a non-general one.
\end{note}

\subsubsection{A subdomain of $\phi$}
In this subsubsection we find an open subset of $X \times X$ contained in the domain of $\phi$ (i.e.\ a ``subdomain'') which is large enough for our purposes. 

More precisely, we prove the following result.

\begin{proposition}\label{prop:domain}
    Define 
    \[
    U_{\phi}:=\left\{(P,Q) \in X \times X \;\left|\; \kappa(P) \neq \kappa(Q), \begin{gathered}
    L_{\kappa(P),\kappa(Q)} \text{ is not bitangent to }\mathcal{B}, \\
    \text{ and } P, Q \text{ are smooth points on }E_{P,Q}\end{gathered} \right.\right\}.
    \]
    Then $U_{\phi}$ is an open subset of $X \times X$ contained in the domain of $\phi$.
\end{proposition}

Recall that the condition that $L_{\kappa(P),\kappa(Q)}$ is not bitangent to $\mathcal{B}$ means precisely that $E_{P,Q}$ is not the union of two exceptional curves. Moreover, when $\ch k \neq 2$, the second condition is equivalent to asking that, if $L_{\kappa(P),\kappa(Q)}$ is tangent to $\mathcal{B}$, then $P,Q \notin \mathcal{R}$. It will follow from the proof of Proposition \ref{prop:domain} that, set-wise, the morphism $\phi$ sends a point $(P,Q)$ in $U_{\phi}$ to the unique smooth point $R$ of $E_{P,Q}$ such that $(R) \sim 2(\iota(P))-(Q)$ on~$E_{P,Q}$ (this is only relevant when $E_{P,Q}$ is singular, for otherwise it is just the definition of $\phi$). 

Before coming to the proof, let us recall the definition of a  \emph{relative arithmetic genus $1$ curve}. We give this in the general setting of schemes, but we will only need it in the case where $S$ is a $k$-variety (see Remark \ref{Rmk: S variety}).

\begin{mydef}\label{def:relcurve}
    Let $\mathcal{C}$ and $S$ be schemes with $S$ locally Noetherian. A proper morphism $f\colon \mathcal{C} \to S$ is a \emph{relative arithmetic genus $1$ curve} if: 
    \begin{itemize}
        \item $f$ is flat;
        \item the fibers are geometrically integral curves;
        \item the fibers have arithmetic genus $1$.
    \end{itemize}
\end{mydef}

For such a morphism $f$, we denote by $\mathcal{C}^0\subset \mathcal{C}$ the (open) smooth locus of $f$.

On such a relative curve we may do operations among sections as in the following proposition, found in \cite[Prop.~2.7]{DR} and  formulated in the language of schemes. 

\begin{proposition}[Deligne--Rapoport]\label{prop:DR}
    Let $f\colon \mathcal{C} \to S$ be a {relative arithmetic genus $1$ curve}. Assume that there exists a section $e\colon S \to \mathcal{C}^0$. There exists then a unique operation $+\colon \mathcal{C}^0 \times_S \mathcal{C} \to \mathcal{C}$ such that, on $S$ and all base changes, for $x \in \mathcal{C}^0(S)$ and $y \in \mathcal{C}(S)$, we have:
    \[
    \mathcal{O}(x+y) \cong \mathcal{O}(x) \otimes \mathcal{O}(y) \otimes \mathcal{O}(e)^{\otimes-1}
    \]
    locally on $S$. (Here $\mathcal{O}(x)$ denotes the line bundle corresponding to the effective Cartier divisor associated to the image of $x$ in $\mathcal{C}$.) This operation makes $\mathcal{C}^0$ into a commutative group scheme, with unity $e$, acting on $\mathcal{C}$.
    When $S=\Spec k$, the operation $+$ is the usual (possibly singular) elliptic curve operation on $\mathcal{C}^0$.
\end{proposition}

\begin{remark}\label{Rmk: S variety}
    In particular, this proposition tells us that, if we have a relative arithmetic genus $1$ curve $f\colon \mathcal{C} \to S$ over a $k$-variety $S$ with a zero-section $e\colon S \to \mathcal{C}$, then there is a morphism $+\colon \mathcal{C}^0 \times_S \mathcal{C}^0 \to \mathcal{C}$ that restricts to the standard group operation of (possibly singular) elliptic curves on the fibers of $f$.
\end{remark}

\begin{proof}[Proof of Proposition \ref{prop:domain}]
Let $\mathcal{W}$ be the following incidence variety:
\[
\mathcal{W}:=\{(R,P,Q) \in X \times U_{\phi} \mid R \in E_{P,Q}\}.
\]
The projection $f:=pr_2\colon  \mathcal{W} \to U_{\phi}$ may be identified with the base change along $U_{\phi} \to \operatorname{Gr}(1,2)$ (sending $(P,Q)$ to $L_{\kappa(P),\kappa(Q)}$) of the cover $u_2:\mathcal{G} \times_{\P^2,\kappa} X \xrightarrow{u \circ pr_1} \operatorname{Gr}(1,2)$, where $u:\mathcal{G}\to \operatorname{Gr}(1,2)$ is the ({flat}) universal Grassmanian family of lines in $\mathbb{P}^2$, consisting of pairs $(L,P)$ where $L$ is an element in $\operatorname{Gr}(1,2)$ and $P\in\mathbb{P}^2$ a point in $L$. Since $\kappa$ is flat \cite[Rem.~4.3.11]{LIU}, both $u_2$ and $f$ are flat by preservation of flatness under base change and composition. Since the fibers of $f$ satisfy the assumptions of Definition \ref{def:relcurve}, this projection makes $\mathcal{W}$ into a relative curve of arithmetic genus~$1$ over $U_{\phi}$.

This relative curve has various sections, among which we have $$\sigma_{\iota(P)}:=(\iota \circ pr_1,id)\colon U_{\phi} \to \mathcal{W} \subset X \times U_{\phi}\text{ and }\sigma_Q:=(pr_2,id)\colon U_{\phi} \to\mathcal{W} \subset  X \times U_{\phi},$$ i.e. ``$(P,Q) \mapsto (\iota(P),P,Q)$'' and ``$(P,Q) \mapsto (Q,P,Q)$'', respectively. Note that these two sections land in the smooth locus  $\mathcal{W}^0$ of the fibration $f$. We may then define the section
\[
\sigma_{2\iota(P)-Q}\colon U_{\phi} \to \mathcal{W}^0, \ \ (P,Q) \mapsto [2]\sigma_{\iota(P)} - \sigma_Q
\]
of $f$ via the sum operation defined in Proposition \ref{prop:DR} (note that to make sense of a linear combination of sections $\sum_{i=1}^r [n_i] \sigma_i$ with $\sum_{i=1}^r n_i=1$, there is no need to specify a zero-section $e$ for the genus-$1$ curve, as the result of the sum is independent of the choice of $e$). Then the composition 
\[
U_{\phi} \xrightarrow{\sigma_{2\iota(P)-Q}} \mathcal{W} \xrightarrow{pr_1} X
\]
extends the rational map $\phi$, as wished.
\end{proof}

\subsubsection{Associated elliptic surface}\label{subsubsec: elliptic surface}
In the remainder of Section \ref{sec:twogeomproc}, we assume char $k\neq2$. Let $P\in X(k)$ be a rational point on $X$. Below we describe how to obtain an associated elliptic surface $\mathcal{E}_P$, and how the rational curves and  point procedures can be viewed on this elliptic surface, following \cite[\S4]{FvL}.

Recall that $\mathcal{R}$ denotes the ramification divisor of $X$.

\begin{mydef}
Given a rational point $P$ on $X$, we define a surface $\mathcal{E}_P$ and curves on it as follows.
\begin{itemize}
\item If $P$ is not contained in $\mathcal{R}$, we set 
$$\mathcal{E}_P := \text{Bl}_{P,\iota\left(P\right)}X
,$$
and we write $E_P$ and $E'_P=E_{\iota(P)}$ for the exceptional divisors in $\mathcal{E}_P$ above $P$ and $\iota(P)$ respectively. 
\item If $P$ is contained in $\mathcal{R}$, set $X' := \text{Bl}_{P}X$. The strict transform in $X'$ of the pencil $|{-K_X-P}|$ on $X$ has a unique base point $P'$ (lying on the exceptional curve above $P$), of multiplicity one. We define 
$$
\mathcal{E}_P := \text{Bl}_{P'}X',
$$
and we write $E_P$ and $E'_P$ for the strict transform in $\mathcal{E}_P$ of the exceptional divisor in $X'$ above $P$ and for the exceptional divisor in $\mathcal{E}_P\to X'$ above $P'$, respectively.
\end{itemize}
We denote by $b\colon \mathcal{E}_P \to X$ the double blow-up. 
\end{mydef}

In both cases, $b$ factors through the blow-up $X'=\text{Bl}_PX\to X$. 
Let $\mathfrak{L}_{\kappa\left(P\right)}\cong\mathbb{P}^1$ be the pencil of lines in $\mathbb{P}^2$ passing through $\kappa\left(P\right)$. The morphism $\kappa$ induces a morphism $\mu_1\colon \mathcal{E}_P\to\text{Bl}_{\kappa(P)}\mathbb{P}^2$. We identify Bl$_{\kappa(P)}\mathbb{P}^2$ with the variety $\{(Q,L)\in\mathbb{P}^2\times\mathfrak{L}_{\kappa(P)}\colon Q\in L\}$, and denote by $\mu_2$ the projection of Bl$_{\kappa(P)}\mathbb{P}^2$ onto the second factor. Let $\nu\colon \mathcal{E}_P\to\mathfrak{L}_{\kappa\left(P\right)}$ be the composition $\mu_2\circ\mu_1$. The maps are represented in the following diagram.

\begin{center}
\begin{tikzcd}
\mathcal{E}_P \arrow[d,"b"] \arrow[rr,bend left=25,"\nu"] \arrow[r, "\mu_1"] & \text{Bl}_{\kappa(P)}\mathbb{P}^2\arrow[d,"\text{blow-up}"]\subset \mathbb{P}^2\times\mathfrak{L}_{\kappa(P)}\arrow[r,"\mu_2"]&\mathfrak{L}_{\kappa(P)} \cong\mathbb{P}^1\\
X\arrow[r,"\kappa"]&\mathbb{P}^2&
\end{tikzcd}
\end{center}

The fiber of $\nu$ above a line $L\in\mathfrak{L}_{\kappa(P)}$ equals the strict transform under the map $b$ of the curve $\kappa^{-1}(L)$ in $|{-K_X}|$, which is an element of $|{-K_{\mathcal{E}_P}}|$. We have
$$
K_{\mathcal{E}_P}^2=
\begin{cases}
\left(b^*(K_X)+E_P+E_{\iota(P)}\right)^2=2-2=0
\text{ if $P \notin \mathcal{R}$,} \\
\left(b^*(K_X)+E_P+2E'_{P}\right)^2=2-2-4+4=0
\text{ if $P \in \mathcal{R}$.}
\end{cases}
$$
By the adjunction formula, the genus of a fiber of $\nu$ above a line $L\in\mathfrak{L}_{\kappa(P)}$ is equal to
$$
    g(\nu^{-1}(L)) = \frac{1}{2}\nu^{-1}(L) \cdot (\nu^{-1}(L) + K_{\mathcal{E}_P} ) +1 = \frac12\left(-K_{\mathcal{E}_P}\cdot 0\right)+1=1,
$$
hence $\nu$ induces a genus-$1$ fibration on $\mathcal{E}_P$. Moreover, if $P\notin\mathcal{R}$, then each fiber intersects both $E_P$ and $E_{\iota(P)}$ with multiplicity 1 (since exceptional curves intersect anticanonical divisors with multiplicity $1$), hence these define two sections of $\nu$; indeed, both $E_P$ and $E_{\iota(P)}$ are mapped isomorphically by $\mu_1$ to the exceptional curve above $\kappa(P)$ in Bl$_{\kappa(P)}\mathbb{P}^2$, which is isomorphic to $\mathfrak{L}_{\kappa(P)}$. We conclude that $\mathcal{E}_P$ is an elliptic surface. Still following \cite[\S4]{FvL}, in this case we choose $E_{\iota(P)}$ to be the zero-section of the Mordell--Weil group of $\mathcal{E}_P$ (which is the Mordell--Weil group of the generic fiber of $\nu$). If $P \in \mathcal{R}$, then, similarly, $E'_P$ defines a section, and we choose it to be the zero-section of the Mordell--Weil group of $\mathcal{E}_P$. The surface $\mathcal{E}_P$ is proper and regular, and it is a \textsl{minimal elliptic surface}, i.e.\ no fiber contains a $\left(-1\right)$-curve: its singular fibers may be of types I$_1$, I$_2$, I$_3$, II, III, IV in Kodaira's classification, as shown in \cite[Tables 1 and 2]{KUW}. 

\begin{remark}\label{rem: elliptic fibers anticanonical curves}
The pencil of fibers of $\nu$ is by construction the strict transform of the pencil of anticanonical curves $E_{P,Q}$ passing through $P$ and $\iota(P)$. In particular, the elliptic operations on a smooth fiber $(\mathcal{E}_P)_t:=\nu^{-1}(t), t \in \P^1$ of the surface $\mathcal{E}_P$ may be identified via $b$ with the elliptic operations on the anticanonical curve $b((\mathcal{E}_P)_t)$.
\end{remark}

\begin{mydef}
For an integer $n \in \Z$, we define the rational map $[n]\colon \mathcal{E}_P \dashrightarrow \mathcal{E}_P$ as the map defined by multiplication by $n$ on the smooth fibers.
\end{mydef} 

For $n=-1$, the rational map $[-1]\colon \mathcal{E}_P \to \mathcal{E}_P$ is actually a morphism by the minimality of $\mathcal{E}_P \to \P^1$. More precisely, this rational map comes from an automorphism (not just an endomorphism) of the smooth fibers, and, as $\mathcal{E}_P$ is a \emph{minimal} proper regular elliptic surface,
it extends to an automorphism of $\mathcal{E}_P$ (see e.g. \cite[Prop. 9.3.13]{LIU})
The first part of the following result shows how one could alternatively define the curve $C_P$ (see Section \ref{subsubsec: making rational curves}) through the surface~$\mathcal{E}_P$. 

\begin{proposition}\label{prop:cpep} Let $P \in X(k)$ be a point. The following hold.
\begin{itemize}
    \item[]{\em (i)} If $P$ is general, the image of the section $[-1]E_{P}\cong\mathfrak{L}_{\kappa(P)}$ under $b$ is mapped to the curve $C_P$.
    \item[]{\em (ii)} If $P$ is not general, the curve $[-1]E_P$ is mapped under $b$ to $P$.
\end{itemize}
\end{proposition}

\begin{proof}
(i) For $P$ very general this follows from \cite[Prop.~4.1]{FvL}. More generally, assume that $P$ is general, lying on $n$ exceptional curves $E_1,\ldots,E_n$ (note that $0 \leq n \leq 3$).
By construction, we have $C_{P}\sim -2K_X-\sum_{i=1}^nE_i$. Since $P$ is general, the exceptional curves $E_i$ do not contain $\iota(P)$, since the curves $E_i$ and $\iota(E_i)$ intersect each other in exactly two points which are both contained in $\mathcal{R}$. It follows that we have 
\begin{align*}
b^{*}(C_{P})&\sim-2K_{\mathcal{E}_P}+2E_P+2E_{\iota(P)}-\sum_{i=1}^nb^*(E_i)\\
&=-2K_{\mathcal{E}_P}+2E_P+2E_{\iota(P)}-\sum_{i=1}^n(\widetilde{E}_i+E_P)\\
&=-2K_{\mathcal{E}_P}+(2-n)E_P+2E_{\iota(P)}-\sum_{i=1}^n\widetilde{E}_i,
\end{align*}
where $\widetilde{E}_i$ is the strict transform of $E_i$ on $\mathcal{E}_P$. The multiplicity of $C_P$ at $P$ is $3-n$ by Note~\ref{Note: multiplicity}, and from \cite[Prop. 4.14]{FvL} it follows that $\iota(P)$ is not contained in $C_P$.
Therefore, the strict transform $\widetilde{C}_P$ of $C_P$ in $\mathcal{E}_P$ has divisor class
\[
 \widetilde{C}_P=b^*(C_P)-(3-n)E_P=-2K_{\mathcal{E}_P}-E_P+2E_{\iota(P)}-\sum_{i=1}^n\widetilde{E}_i.
\]
Since the $\widetilde{E}_i$'s are vertical divisors on the elliptic surface $\mathcal{E}_P$ (see \cite[Tables 1,2]{KUW}) and so is $-K_{\mathcal{E}_P}$, we have $\widetilde{C}_P=[-1]E_P$, as claimed. 

(ii) If $P$ is not general, then it is contained in $\mathcal{R}$ or it is a generalized Eckardt point. In the first case, we have that $E_P$ is a vertical divisor on the elliptic surface $\mathcal{E}_P$ (see \cite[Table~2]{KUW}) that intersects the zero-section $E'_P$. Thus we find $[-1]E_P=E_P$ and $b([-1]E_P)=b(E_P)=P$.
The second case follows from \cite[Thm.~4.13]{FvL}.
\end{proof}

We will now relate the elliptic surface $\mathcal{E}_P$ to the map $\phi$ constructed in Section \ref{subsubsec: propagating points}.

Note that the involution $\iota\colon X \to X$ extends to an involution $\mathcal{E}_P \to \mathcal{E}_P$ that commutes with $\nu$, which we still denote by $\iota$. 

\begin{proposition}\label{prop:comm}
    Let $P \in X(k)$ be any rational point. Let $T_P$ be the section $E_P$ of $\mathcal{E}_P$ when $P \notin \mathcal{R}$, and $T_P$ is equal to the zero-section $E'_P$ when $P \in \mathcal{R}$. The following diagrams commute:
    \[
    \begin{tikzcd}
        \mathcal{E}_P \arrow[d, "b"] \arrow[r, "{[2]\circ \iota - T_P}", dashed] & \mathcal{E}_P \arrow[d, "b"] \\
        X \arrow[r, "{\phi(-,P)}", dashed]                      & X,                           
    \end{tikzcd} \quad 
    \begin{tikzcd}
        \mathcal{E}_P \arrow[d, "b"] \arrow[r, "{[-1]}"] & \mathcal{E}_P \arrow[d, "b"] \\
        X \arrow[r, "{\phi(P,-)}", dashed]                      & X.                           
    \end{tikzcd}
    \]
\end{proposition}

\begin{proof}
    We prove the commutativity of the first one as the second one is proved analogously. It suffices to verify the commutativity for a Zariski-general fiber $\left(\mathcal{E}_{P}\right)_{t}$ of the map $\nu\colon\mathcal{E}_P \to \P^1$. This (elliptic) fiber is identified via $b$ with a Zariski-general element of the pencil $|{-K_X-P}|$ (see also Remark \ref{rem: elliptic fibers anticanonical curves}), and the origin of $\left(\mathcal{E}_{P}\right)_{t}$ is identified via $b$ with the point $\iota(P)$. Note that the map $[2]\colon  \mathcal{E}_P \to \mathcal{E}_P$ restricts by definition to multiplication by $2$ on $\left(\mathcal{E}_{P}\right)_{t}$, and thus ${[2]\circ \iota - T_P}$ restricts to the map $Q \mapsto [2]\iota(Q)-P$ on $\left(\mathcal{E}_{P}\right)_{t}$. On the other hand, ${\phi(-,P)}$ restricts on $\left(\mathcal{E}_{P}\right)_{t}$ to the map sending a point $Q$ to the unique point $R$ of $\left(\mathcal{E}_P\right)_t$ such that $(R) \sim 2 (\iota(Q))- (P)$. 
    The last condition is equivalent to $(R)- (\iota(P)) \sim 2(\iota(Q))- (P)- (\iota(P))$, which is the same as $R=[2]\iota(Q)-P$ since $\iota(P)$ is the origin of the elliptic curve $\left(\mathcal{E}_{P}\right)_{t}$.
\end{proof}

\begin{remark}\label{rem:inv}
It follows from the commutativity of the second diagram in Proposition~\ref{prop:comm} that the domain of the rational map ${\phi(P,-)}\colon  X  \dashrightarrow X$ contains $X \setminus \{P, \iota(P)\}$, and that ${\phi(P,-)}$ induces an involution on 
\begin{equation*}
    X \setminus b(E_P,E_P', [-1] E_P)= X \setminus \{P,\iota(P), C_P\},
\end{equation*}
since $b([-1]E_P)=C_P$ (resp.\ $b([-1]E_P)=P$) when $P$ is general (resp.\ otherwise) by Proposition \ref{prop:cpep}.
\end{remark}

\begin{remark}\label{Rmk:domain restriction}
    In Proposition \ref{prop:comm} and Remark \ref{rem:inv} we proved that the restriction $\phi(P,-)$ is well-defined on $X \setminus \{P, \iota(P)\}$. Notice on the other hand how, through Proposition \ref{prop:domain} alone, we could a priori only infer that $\phi(P,-)$ is a morphism on $U_{\phi} \cap (\{P\} \times X)$, which is in general smaller than $X \setminus \{P, \iota(P)\}$. For instance, it does not contain any point on any of the lines passing through~$P$. This phenomenon is consistent with the fact that the domain of the restriction of a rational map may in general be larger than the restriction of the domain.
\end{remark}

We denote by $\phi_P$ (resp.\ $\phi^P$) the rational map $\phi(P,-)$ (resp.\ $\phi(-,P)$). This notation serves also the purpose of making clearer (notationwise) the domain issue raised in Remark \ref{Rmk:domain restriction}: when $Q$ (resp.\ $P$) lies in the domain of $\phi(P,-)$ (resp.\ $\phi(-,Q)$) but $(P,Q)$ does not lie in the domain of $\phi$, we will only use the notation $\phi_P(Q)$ (resp.\ $\phi^Q(P)$), and not the ambiguous notation $\phi(P,Q)$.

\subsubsection{Key proposition}
In this subsubsection we give a proposition that will be repeatedly used in the proof of Theorem \ref{thm:main1}. Recall that $P_0 \in X(k)$ is fixed and very general. Let $C_{P_0}$ be the corresponding curve in Definition \ref{def:dp2curves}.

\begin{proposition}\label{prop:drm}
Let $D$ be an irreducible curve on $X$.
Then the rational map 
$$
\psi= \phi |_{C_{P_0}\times D}\colon  C_{P_0}\times D \dashrightarrow X
$$
is dominant.
\end{proposition}

\begin{proof}
    We may assume that $k$ is algebraically closed.
    Note that $U_{\phi} \cap (C_{P_0} \times D) \neq \emptyset$, thus $\psi$ is well-defined. Moreover, one checks that:
    \begin{itemize}
        \item $U_{\phi}\cap (C_{P_0} \times \{Q\})\neq \emptyset$ for all $Q \in D$;
        \item $U_{\phi}\cap (\{P\} \times D)\neq \emptyset$ for all $P \in C_{P_0}$ unless $D$ is a line, and $P \in (D \cup \iota(D))$.
    \end{itemize}
    We denote by $L$ the (finite) set $(D \cup \iota(D)) \cap C_{P_0}$ if $D$ is a line, while $L:= \emptyset$ otherwise. (Note that, since $P_0$ is very general, it is not contained in $L$.) Proposition \ref{prop:domain} then allows us to restrict $\psi$ (equiv.\ $\phi$) on $C_{P_0} \times \{Q\}$ for all $Q \in D$  and on $\{P\} \times D$ for all $P \in C_{P_0} \setminus L$, which we will both do repeatedly throughout the proof.

    Since $C_{P_0} \times D$ is irreducible, so is the closed image of $\psi$. To show that $\psi$ is dominant, assume by contradiction that its closed image is a one-dimensional curve~$C$.

    For a Zariski-general $Q \in D$ we have, by Proposition \ref{prop:comm}, a commutative diagram:
    \begin{equation*}
        \begin{tikzcd}
            \mathcal{E}_Q \arrow[d,"b"] \arrow[r, "{[2]\circ \iota - T_Q}", dashed] & \mathcal{E}_Q \arrow[d,"b"] \\
            X \arrow[r, "\phi^Q", dashed]                      & X        
        \end{tikzcd}
    \end{equation*}
     (Recall that $\phi^Q\colon X \dashrightarrow X$ is the restriction $\phi(-,Q)$.)
By \cite[\S4]{KUW}, all the singularities of the singular fibers of $\mathcal{E}_Q \xrightarrow{\nu} \mathbb{P}^1$ lie on the strict transform $\widetilde{\mathcal{R}}$ under $b$ of the ramification curve $\mathcal{R}$. In particular, $\mathcal{E}_Q \xdashrightarrow{[2]\circ \iota -T_Q} \mathcal{E}_Q$ is well-defined and étale outside $\widetilde{\mathcal{R}}$. 
    
Since $P_0$ is very general, the curve $C_0 := {C}_{P_0}$ has a (triple) singularity at the point~$P_0$. Since $P_0$ is not blown-up under $b$ ($Q$ and $\iota(Q)$ are Zariski-general on some curve, and thus different from $P_0$ in particular), it follows that the strict transform $C_1$ of $C_0$ in $\mathcal{E}_Q$ has a singularity at the point $P_1:=b^{-1}(P_0)$. Moreover, this point does not lie on $\widetilde{\mathcal{R}}$ (as $P_0 \notin \mathcal{R}$ by assumption), thus $[2]\circ \iota -T_Q$ is étale at $P_1$. Since the image of a curve singularity under an unramified map is still a curve singularity, it follows that the irreducible curve $C_2:=([2]\circ \iota -T_Q)(C_1)$ has a singularity at $P_2:=([2]\circ \iota -T_Q)(P_1)$. 

Note that $b(C_2)$ is one-dimensional, as otherwise the irreducible curve $C_2$ would be contained in the exceptional divisor for $b$ and would thus be an open subvariety of a smooth rational curve, contradicting that it has a singularity. Thus $b(C_2)$ is a curve in $X$ with a singularity at $b(P_2)$. By commutativity of the diagram above, this translates into saying that $\phi(C_0,Q)$ is a curve with a singularity at $\phi^{P_0}(Q)=\phi(P_0,Q)$ (since $P_0 \notin L$, $(P_0,Q)$ lies in the domain of $\phi$). Since $\phi(C_0,Q) \subset \phi(C_0,D) \subset C$, it follows that $\phi(C_0,Q)$ is an open subcurve of $C$, and thus the latter has a singularity at $\phi(P_0,Q)$ as well. 

We now divide into two cases. In both cases, we will derive a contradiction to the claim that the image of $\psi$ is a one-dimensional curve $C$ by showing that $C$ has too many singularities.

    \textbf{Case 1:} $D \neq C_{P_0}$. Note that $\phi({P_0},D) \subset C$, and since $\phi({P_0},D)$ is one-dimensional by Remark \ref{rem:inv}, this implies that $\phi({P_0},D)$ is an open subcurve of $C$. Since $C$ is singular at $\phi(P_0,Q)$ for all Zariski-general $Q \in D$, we deduce that $C$ is singular at all but finitely many points of the one-dimensional open subcurve $\phi(P_0,D)$, which leads to the desired contradiction.

   \textbf{Case 2:} $D=C_{P_0}$. In this case, $D$ has the singularity $P_0$. Thus one may argue using the second diagram of Proposition \ref{prop:comm} to prove that $C$ is singular at $\phi(Q,P_0)$ for all Zariski-general $Q \in C_{P_0}$. 
   
We claim now that $\phi(C_{P_0},P_0)$ is one-dimensional. In fact, assume by contradiction that it is not, then we would have that $\phi(C_{P_0},P_0)$ is constantly equal to some point $R$ (when defined). In other words, for all Zariski-general $Q \in C_{P_0}$ we would have:
   \begin{enumerate}
       \item $R$ belongs to $E_{Q,P_0}$;  and 
       \item  $2(\iota(Q))-(P_0) \sim (R)$ on $E_{Q,P_0}$.
   \end{enumerate}
Condition (1) implies that $R=P_0$ or $R=\iota(P_0)$. Remembering that, since $Q$ is contained in $C_{P_0}$, we have $(Q) \sim 2(\iota(P_0)) - (P_0)$ and thus $(\iota(Q)) \sim 2(P_0) - (\iota(P_0))$, we deduce (using Condition (2)) with a small computation that $2(P_0) - 2 (\iota(P_0)) \sim 0$ when $R = P_0$ and that  $3(P_0) - 3(\iota(P_0)) \sim 0$ when $R = \iota(P_0)$. In either scenario, we deduce that the section $E_{P_0}$ is a non-trivial torsion section of the elliptic surface $\mathcal{E}_{P_0}$. Recalling that $P_0$ is very general, we have that the elliptic surface $\mathcal{E}_{P_0}$ is the blow-up of the del Pezzo surface $\operatorname{Bl}_{P_0}X$ of degree 1 in the base-point $\iota(P_0)$ of its anticanonical pencil. Such elliptic surfaces are known to have a torsion-free Mordell-Weil group \cite[Thm.~7.4]{SS19}, leading to the desired contradiction and thus proving that $\phi(C_{P_0},P_0)$ is one-dimensional.

We may now conclude as in Case 1, since we proved that $C$ is singular at all points of the one-dimensional (open) subcurve $\phi(C_{P_0},P_0)$.
\end{proof}

\section{Proof of main results}
Let $X$ be a del Pezzo surface of degree $2$ over a field $k$. As in \cite{DS}, we will iteratively produce covers $f_n\colon  Z_n \dashrightarrow X$ with desirable properties relative to the goal of showing that the rational points of $X$ are abundant and well-distributed. We assume again that $P_0 \in X\left(k\right)$ is a very general point (Definition~\ref{def:gen}).

\begin{note} \label{note:poschar}
Henceforth, we will assume that $\ch k = 0$. Emulating our strategy in positive characteristic, one encounters issues regarding separability of morphisms; these obstacles appear to be surmountable, at least when $\ch k \neq 2$, but we do not explore this here.
\end{note}

\subsection{New unirationality cover}
\begin{mydef}
Set $Z_1 :=\mathbb{P}^1_k$, and let $f_1\colon  Z_1 \rightarrow X$ be the desingularisation $\mathbb{P}^1 \rightarrow C_{P_0}$.

For $n = 2,3$, define $Z_n := (\P^1)^n$, and let $f_n\colon  Z_n \dashrightarrow X$ be the rational map

$$
f_n\colon Z_n = Z_1 \times Z_{n-1} \xdashrightarrow{\left(f_1,f_{n-1}\right)} X \times X \xdashrightarrow{\phi} X. 
$$

\end{mydef}
\begin{proposition} \label{prop:rat}
$Z_2$ is rational, and the rational map 
$$f_2\colon Z_2\xrightarrow{\left(f_1,f_1\right)} C_{P_0}\times C_{P_0}\xdashrightarrow{\phi}X$$ 
is dominant.
\end{proposition}
\begin{proof}
Follows from Proposition \ref{prop:drm}.
\end{proof}

\begin{remark}
Note that Proposition \ref{prop:rat} implies that $X$ is unirational. The cover $f_2\colon  Z_2 \dashrightarrow X$ differs from the one used to prove unirationality in \cite{STVA}.
\end{remark}

\begin{remark}
Interestingly, the analogue of this construction for a cubic surface $W$ does not imply unirationality, since, in that case, $C_{P_0}$ is the intersection of $W$ with the tangent plane $T_{P_0}W$ and $\phi_W(P,Q)$ is defined as the third intersection point of $L_{P,Q} \cap W$. In particular, the lines defining the points in $\phi_W(C_{P_0}, C_{P_0})$ are all contained in $T_{P_0}W$, hence $\phi_W(C_{P_0}, C_{P_0}) \subset T_{P_0}W \cap W = C_{P_0}$.
\end{remark}

\subsection{Hilbert-unirationality}
Recall that $f_3\colon  Z_3 \dashrightarrow X$ is given by 
$$
f_3\colon  Z_3 = Z_1 \times  Z_2\xdashrightarrow{(f_1,f_2)} X\times X\xdashrightarrow{\phi} X.
$$ 
\begin{proposition}\label{GIGF}
$Z_3$ is rational, and $f_3$ is dominant with geometrically integral generic fiber. In particular, $X$ is Hilbert-unirational.
\end{proposition}

Before coming to the proof, we present a rational variant of Stein factorization. 
\begin{lemma}[Rational Stein Factorization]\label{LemStein}
    Let $k$ be a field and $f\colon W \dashrightarrow V$ a dominant rational map of integral $k$-varieties. Then there exists a finite integral normal cover $V' \to V$ and a decomposition of $f$ as
    \begin{equation}\label{Eq:RationalSteinDec}
        f\colon  W \xdashrightarrow{f'} V' \to V,
    \end{equation}
    for some dominant rational map $f'$ with geometrically integral generic fiber.
\end{lemma}
\begin{proof}
    Let $L$ be the algebraic closure of $k(V)$ in $k(W)$. This is a finite extension of $k(V)$. Let $V'\to V$ be the relative normalization of $V$ in $L$ (\cite[Def.~4.1.24]{LIU}). Then $k(V') = L$, the embedding $k(V') \subset k(W)$ defines a dominant rational map $f'\colon W \dashrightarrow V'$, and $f$ factors as in \eqref{Eq:RationalSteinDec} by construction. The generic fiber of $f'$ is a $k(V')$-variety with function field $k(W)$. Noting that $k(V')$ is algebraically closed in $k(W)$ and that a variety is geometrically integral if and only if its base field is algebraically closed in its function field \cite[Exercise~3.2.12]{LIU}, this proves the lemma.
\end{proof}

In the following proof, it will prove useful to talk about \emph{branch divisors} of generically finite rational maps. We give the following definition.

\begin{mydef}\label{def:branch}
    Let $\psi\colon X_1 \dashrightarrow X_2$ be a generically finite rational map of $k$-varieties with $X_2$ smooth. Let $X'_1 \to X_2$ be the relative normalization of $X_2$ in $k(X_1)$. We define the {\em branch locus} of $\psi$ to be the branch locus of the finite cover $X'_1 \to X_2$.
\end{mydef}

Since $X_2$ is smooth and $X'_1$ is normal, the branch locus is always of pure codimension $1$ by the Zariski--Nagata purity theorem \cite[Tag 0BMB]{SP}. Therefore, we will also refer to the branch locus as the {\em branch divisor}.

Note that, if $X_1 \xdashrightarrow{\psi_1} X_2 \xdashrightarrow{\psi_2} X_3$ is a composition of generically finite maps to smooth varieties, then the branch divisor of $\psi_2$ is contained in the branch divisor of $\psi_2 \circ \psi_1$.

\begin{proof}[Proof of Proposition \ref{GIGF}]
The rationality of $Z_3=Z_2 \times Z_1$ follows from that of $Z_2$ and~$Z_1$. That $f_3$ is dominant follows from the fact that $f_2$ is dominant (Proposition \ref{prop:rat}) and so is (evidently) $\phi|_{C_{P_0} \times X}\colon  C_{P_0} \times X \dashrightarrow X$.

It remains to prove that the generic fiber of $f_3$ is geometrically integral. To do so, we will assume without loss of generality that $k$ is algebraically closed.

Let us assume by contradiction that the generic fiber of $f_3$ is not geometrically integral. Then by Lemma \ref{LemStein}, we get a factorization
\[
f_3\colon  Z_3 \xdashrightarrow{f_3'} X' \to X
\]
for some rational map $f_3'$ and some {\em non-trivial} finite integral normal cover $X' \to X$. We thus get a commutative diagram:
\[
\begin{tikzcd}
Z_3=\mathbb{P}^1\times Z_2\arrow[rd, labels=below left, "{(f_1,f_2)}",dashed] \arrow[rrd, "f_3'", dashed] &  &   \\
       & C_{P_0}\times X  \arrow[rd, "\phi", labels=below left, dashed] & X' \arrow[d] \\
    &                                             & X
\end{tikzcd}
\]
Let $U\subset Z_3$ be the domain of $f_3'$, and take any point $Q' \in \pi(U)\subset \mathbb{P}^1$, where $\pi\colon Z_3\longrightarrow \mathbb{P}^1$ is the projection on the first coordinate. Let $\phi_Q\colon X \dashrightarrow X$ be the specialized rational map $\phi(Q,-)$, where $Q=f_1(Q')\in C_{P_0}$. We may specialize the diagram above to obtain:
\[
\begin{tikzcd}
Z_2 \arrow[rd,labels=below left, "f_2", dashed] \arrow[rrd, "{f_3'(Q',-)}", dashed] & &              \\
 & X \arrow[rd, labels=below left,"\phi_Q", dashed] & X' \arrow[d] \\
& & X           
\end{tikzcd}
\]
In other words, $X' \to X$ is a subcover of $Z_2 \xdashrightarrow{f_2} X 
 \xdashrightarrow{\phi_Q} X$. The former is an irreducible finite cover of degree $\geq 2$, and is thus necessarily ramified as $X$ is algebraically simply connected \cite[Cor.~4.18]{D} and $k$ is algebraically closed. Let $B' \subset X$ be its ({\em non-trivial}) branch divisor. Then $B'$ must be contained in the branch divisor of $\phi_Q \circ f_2: Z_2 \dashrightarrow X$. 
 
 We will reach a contradiction by showing that there is no divisor on $X$ contained in the intersection of the branch divisors of $\phi_Q \circ f_2: Z_2 \dashrightarrow X$ for all $Q\in f_1(\pi(U))$. Assume by contradiction that $D \subset X$ is one such divisor, which we may assume to be irreducible. 

Let $B$ be the branch divisor of $Z_2 \xdashrightarrow{f_2} X$, and take $Q\in f_1(\pi(U))$. By Remark~\ref{rem:inv}, the map $\phi_Q$ is a (regular) involution on $U_1=X \backslash\left\{Q, \iota(Q), C_Q\right\}$. It follows that $\phi_Q\left(D \cap U_1\right)$ is contained in $B \cap U_1$. Choosing $Q\in f_1(\pi(U))$ sufficiently general, the curve $C_Q$ intersects both $D$ and $B$ properly. Indeed, for a very general $Q$, the curve $C_Q$ has a unique singularity at the point $Q$. Thus the map $Q \mapsto C_Q$ is injective on the set of such points $Q$. Since $C_Q$ is irreducible (when defined) and so is $D$, we have that $D \cap C_Q$ is one-dimensional if and only if $C_Q=D$ (same for $B$). Because of the injectivity above, this can happen for at most one point $Q$. This implies that $D \cap U_1$ and $B \cap U_1$ are dense in $D$ and $B$ respectively. Then it follows that for general $Q$, we have $\phi_Q\left(D\right) \subset B$.

Since $f_1(\pi(U))$ is open in $C_{P_0}$, this holds for all Zariski-general $Q \in C_{P_0}$, which implies that the closed image of the composition
\[
C_{P_0} \times D  \hookrightarrow X \times X \xdashrightarrow{\phi} X
\]
is contained in $B$. In particular, it is one-dimensional, contradicting Proposition \ref{prop:drm}.\end{proof}

By Proposition \ref{prop:goodcover}, we thus obtain the following corollary. 

\begin{corollary}
If $k$ is a Hilbertian field, then any degree-$2$ del Pezzo surface $X$ with a general $k$-rational point satisfies the Hilbert property.
\end{corollary}

\subsection{Split-unirationality}
Recall that we assume $\operatorname{char} k=0$.
Define $Z_6:=Z_3\times Z_3$, and let $f_6\colon Z_6\dashrightarrow X$ be the map

$$
f_6\colon Z_6 = Z_3 \times Z_3 \xdashrightarrow{(f_3,f_3)} X \times X \xdashrightarrow{\phi} X.
$$

We follow here the reasoning in \cite{DS}. Namely, the main result of this subsection is the following proposition, which shows that $f_6$ satisfies the assumptions of Proposition \ref{prop:ssfcover}, thus showing that $X$ is split-unirational (indeed, note that $Z_6=(\P^1)^6$ is rational). By Proposition \ref{prop:goodcover}, this implies that $X$ satisfies weak weak approximation.

\begin{proposition}\label{prop:su}
    There is an open subset $W\subset Z_6$ such that
    \begin{enumerate}[label={\em (\roman*)}]
        \item $f_6$ restricts to a morphism on $W$;
        \item $f_6|_W$ is smooth with split or empty fibers;
        \item $f_6(W)=X$.
    \end{enumerate}
In particular, $X$ is split-unirational.
\end{proposition}

The reader might notice that (ii) and (iii) combined may be summarized as ``$f_6|_W:W \to X$ is smooth with split fibers''. We made this separation because we will prove that there exists such a $W$ by guaranteeing one condition at a time: we produce $W_1$, $W_2$, and $W=W_3$ that satisfy (i), (i)-(ii), and (i)-(iii), respectively.

\subsubsection{Condition {\em (i)}.}

Recall that $U_{\phi} \subset X \times X$ was the subdomain of $\phi$ defined in Proposition \ref{prop:domain}. Letting $V_3 \subset Z_3$ be the domain of the rational map $f_3$, clearly $f_6$ restricts to a morphism on $W_1:=(V_3 \times V_3) \cap (f_3,f_3)^{-1}(U_{\phi})$.

\subsubsection{Conditions {\em (i)-(ii)}}

We first define an open subset $U_{\mathrm{inv}} \subset U_{\phi}$ that is endowed with a natural involution and on which we will prove (Corollary \ref{cor:smoothGIF}) that $\phi$ is smooth with geometrically integral fibers. Let $U_0 \subset X$ be the complement of the ramification divisor $\mathcal{R}$ and the lines on $X$.
Note that the curve $C_P$ is well-defined for $P \in U_0$ and that $\{(P,Q) \in U_0 \times X \mid Q \in C_P\}$ is closed in $U_0 \times X$.
Thus
\[
U_{\mathrm{inv}}:=\{(P,Q) \in U_{\phi} \mid P \in U_0, Q \notin C_P\}
\]
is open in  $U_{\phi}$, equiv.\ in $X \times X$.

For $i=1,2$, we denote by $pr_i\colon U_{\mathrm{inv}}\longrightarrow X$ the projection on the $i^{\tiny\text{th}}$ coordinate.

\begin{lemma}\label{LemSmooth}
    The map $(pr_1,\phi)\colon U_{\mathrm{inv}} \to X \times X$ has image equal to $U_{\mathrm{inv}}$, and the morphism $(pr_1,\phi)\colon U_{\mathrm{inv}} \to U_{\mathrm{inv}}$ is an involution that makes the following diagram commute:
    \[
    \begin{tikzcd}
        U_{\mathrm{inv}} \arrow[rd, "\phi"] \arrow[rr, "{(pr_1,\phi)}", "\sim"'] &   & U_{\mathrm{inv}} \arrow[ld, "pr_2"] \\
        & X &                     
    \end{tikzcd}
    \]
\end{lemma}
\begin{proof}
    Fix $P \in U_0$. For all $Q\notin \{P, \iota(P)\}\cup C_P$ we have $L_{\kappa(P),\kappa(Q)}=L_{\kappa(P),\kappa(\phi(P,Q))}$, and $\phi^Q(P)\notin\{P, \iota(P)\}\cup C_P$ by Remark \ref{rem:inv}. When $(P,Q) \in U_{\mathrm{inv}}$, $\phi^Q(P)=\phi(P,Q)$. Since $P$ was arbitrary, it follows that $(pr_1,\phi)(U_{\mathrm{inv}})$ is contained in $U_{\mathrm{inv}}$. Let us now show that $(pr_1,\phi)$ is an involution. By definition, we have $(pr_1,\phi)(P,Q)=(P,R)$, where $R$ is the unique smooth point on $E_{P,Q}$ such that $(R) \sim 2(\iota(P)) - (Q)$. It follows that $(Q) \sim 2(\iota(P)) - (R)$, hence $(pr_1,\phi)(P,R)=(P,Q)$, i.e.\ $(pr_1,\phi)$ is an involution, as claimed. The commutativity of the diagram is clear.
\end{proof}

\begin{corollary}\label{cor:smoothGIF}
$\phi\colon U_{\mathrm{inv}} \longrightarrow X$ is smooth with geometrically integral fibers.
\end{corollary}
\begin{proof}
By the commutativity of the diagram in Lemma \ref{LemSmooth}, this follows directly from the fact that $pr_2\colon U_{\mathrm{inv}} \subset X \times X \to X$ satisfies the same properties. Indeed, $pr_2$ is smooth as the projection $X \times X \to X$ is, and the fibres of $pr_2$ are non-empty open subsets of $X$ and are thus smooth and geometrically integral as $X$ is.
\end{proof}

Let now $U \subset X$ be an open subset over which $f_3|_{V_3}$ is smooth with geometrically integral fibers: since $V_3$ is smooth and the generic fiber of $f_3$ is geometrically integral (Proposition \ref{GIGF}), such $U$ exists by generic smoothness in characteristic zero \cite[Cor.~III.10.7]{HAR} and spreading out \cite[Thm.~9.7.7]{EGA}. 

Define $W_2\subset W_1$ as the open subset $(V_3 \times V_3) \cap (f_3,f_3)^{-1}(U_{\mathrm{inv}} \cap (U \times U))$. 
Note that in the composition 
\[
f_6\colon W_2 \xrightarrow{(f_3,f_3)} U_{\mathrm{inv}} \cap (U \times U) \xrightarrow{\phi} X,
\]
the first morphism is smooth with geometrically integral fibers by construction of $U$, while the second morphism is smooth with geometrically integral fibers by Corollary \ref{cor:smoothGIF}. The following lemma then proves that condition (ii) holds.

\begin{lemma}\label{LemLiu}
    Let $f\colon  X_1 \xrightarrow{g} X_2 \xrightarrow{h} X_3$ be a composition of morphisms of algebraic $k$-varieties such that $g$ is flat and $g$ and $h$ have geometrically integral fibers. Then $f$ has geometrically integral fibers. Moreover, if both $g$ and $h$ are smooth, then so is $f$.
\end{lemma}
\begin{proof}
    We use the language of schemes (so points are scheme-theoretic points). For $t \in X_3$, we have that $g|_{f^{-1}(t)}\colon f^{-1}(t) \to h^{-1}(t)$ is a flat morphism with geometrically integral fibers over a geometrically integral $k(t)$-variety. Applying \cite[Prop.~4.3.8]{LIU} to the base change of $g|_{f^{-1}(t)}$ to $\overline{k(t)}$, we deduce that $f^{-1}(t)$ is a geometrically integral $k(t)$-variety. For smoothness of the composition, see \cite[Lem.~29.34.4,~Tag~01VA]{SP}.
\end{proof}

\subsubsection{Conditions {\em (i)-(iii)}}
After restricting $U$ if necessary and replacing $V_3$ by $V_3 \cap f_3^{-1}(U)$, we may assume $f_3(V_3)=U$. Let $W_3$ be defined as $W_2$ after replacing $U$ and $V_3$ by their restrictions (which we call $U$ and $V_3$ again, respectively), i.e.\ $W_3 = (V_3 \times V_3) \cap (f_3,f_3)^{-1}(U_{\mathrm{inv}}\cap(U\times U))$. Then the first morphism in the following composition is surjective:
\[
f_6\colon W_3 \xrightarrow{(f_3,f_3)} U_{\mathrm{inv}} \cap (U \times U) \xrightarrow{\phi} X.
\]

To verify that the composition is surjective, it is then enough to prove that $\phi\colon U_{\mathrm{inv}} \to X$ is surjective and that all fibers of $\phi|_{U_{\mathrm{inv}}}$ intersect $U \times U$. We verify these two properties in Corollary \ref{cor:surjective} and Corollary \ref{cor:slanted}, respectively.

\begin{lemma}\label{LemSurjectivity}
    For every $Q \in X$, there exists a non-empty open subset $U_Q \subset X$ such that $(P,Q)$ is contained in ${U_{\mathrm{inv}}}$ for every $P \in U_Q$. 
\end{lemma}
\begin{proof}
    Assume by contradiction that there exists $Q \in X$ (which we fix for the remainder of the proof) such that we have $(P,Q) \notin {U_{\mathrm{inv}}}$ for every $P \in X$. Indeed, it is sufficient to show that this is false, since the set of $P \in X$ such that $(P,Q) \in U_{\text{inv}}$ is open in $X$. Note that, for a Zariski-general $P \in X$, we certainly have that $(P,Q) \in U_{\phi}$ and $P \in U_0$. In particular, our contradiction assumption then implies that $Q$ is contained in $C_P$ for all Zariski-general $P \in X$. Fix such a $P$. From Remark~\ref{rem:phidef}~(1) it follows that $(Q) \sim 2(\iota(P))-(P)$ on $E_{P,Q}$. But this also holds for $\iota(P)$, since $P$ is Zariski-general if and only $\iota(P)$ is. Then on $E_{P,Q}=E_{\iota(P),Q}$ we have 
    $$
    2(P)-(\iota(P)) \sim (Q)\sim 2(\iota(P))-(P) ,
    $$
    so we find $3((P)-(\iota(P))\sim 0$, hence $P=\iota(P)+T$, where $T$ is a 3-torsion point of the Jacobian of $E_{P,Q}$. Setting $Q$ as the origin of $E_{P,Q}$, we have that $\iota(P)=[2]P$, thus $P=-T$ is 3-torsion. Since this holds for all $P'$ in a Zariski-open set $U$ in $X$, this gives a Zariski-open set $U'=U\cap E_{P,Q}$ of $E_{P,Q}$, such that all points in $U'$ are 3-torsion on $E_{P',Q}=E_{P,Q}$. Thus $E_{P,Q}$ has infinitely many 3-torsion points, giving a contradiction. 
\end{proof}

\begin{corollary}\label{cor:surjective}
    $\phi\colon {U_{\mathrm{inv}}} \longrightarrow X$ is surjective.
\end{corollary}
\begin{proof}
Combining Lemma \ref{LemSmooth} and Lemma \ref{LemSurjectivity}, it directly follows that the fibers of $\phi\colon {U_{\mathrm{inv}}} \longrightarrow X$ are non-empty.
\end{proof}

The following lemma proves that no fiber of $\phi$ is contained in $(D\times X) \cup (X\times D)$ for any proper closed subvariety $D$ of $X$ (we say that the fibers are \textsl{slanted} in $X \times X$).

\begin{lemma}[Slantedness]\label{LemSlantedness}
    For every proper closed subvariety $D\subsetneq X$, and $R \in X$ a point, there exists a point $(P,Q) \in {U_{\mathrm{inv}}}\cap \phi^{-1}(R)$ with $P,Q \notin D$. 
\end{lemma}
\begin{proof}
    Fix $R \in X$ and $D \subsetneq X$ as in the statement. Let $U_R\subset X$ be the non-empty open set of Lemma \ref{LemSurjectivity}. By Proposition \ref{prop:comm}, the rational map $\phi(-,R)\colon X \dashrightarrow X, P \mapsto 2(\iota(P))-(Q)$ is dominant. In particular, for Zariski-general $P\in U_R$, we have that $Q=\phi(P,R)=2(\iota(P))-(R) \notin D$ and clearly $P \not\in D$. Since $\phi(P,Q)$ is now equal to $R$, this finishes the proof. 
\end{proof}
\begin{corollary}\label{cor:slanted}
    All fibers of $\phi\colon {U_{\mathrm{inv}}} \longrightarrow X$ intersect $U \times U$.
\end{corollary}
\begin{proof}
    Take $D=X \setminus U$ in Lemma \ref{LemSlantedness}.
\end{proof}
Combining the above and applying Proposition \ref{prop:ssfcover}, we deduce Proposition~\ref{prop:su}.

\subsubsection{Proof of the main theorem}

We already phrased the argument in the paragraph above Proposition \ref{prop:su}, but we repeat it here for the reader's convenience.

\begin{proofmainthm}
Proposition \ref{prop:su} shows that $X$ is split-unirational. Thus, according to Proposition \ref{prop:goodcover}, we deduce that it satisfies weak weak approximation. \qed
\end{proofmainthm}
 

\begin{thebibliography}{1}
 
\bibitem{CTS} J.-L. Colliot-Th\'el\`ene, J.-J. Sansuc, {\em Principal homogeneous spaces under flasque tori: applications.} J. Algebra 106 (1987), no. 1, 148--205.

\bibitem{CTSSK} J.-L. Colliot-Th\'el\`ene, J.-J. Sansuc, P. Swinnerton-Dyer,  {\em Intersections of two quadrics and Ch\^atelet surfaces. I.} Journal für die reine und angewandte Mathematik 373 (1987), 37--107.

\bibitem{CTSk} J.-L. Colliot-Th\'{e}l\`ene, A. Skorobogatov, {\em The Brauer–Grothendieck group}. Ergebnisse der Mathematik und ihrer Grenzgebiete. 3. Folge. A Series of Modern Surveys in Mathematics, 71. Springer, Cham, 2021.

\bibitem{CZ} P. Corvaja, U. Zannier, {\em On the Hilbert property and the fundamental group of algebraic varieties.} Math. Z. 286 (2017), no. 1-2, 579–602.

\bibitem{D} O. Debarre, {\em Higher-dimensional algebraic geometry}, Universitext, Springer-Verlag, New York, 2001.

\bibitem{DR} P. Deligne,  M. Rapoport, {\em Les sch\'{e}mas de modules de courbes elliptiques}, Modular functions of one variable, {II}, {P}roc. {I}nternat. {S}ummer {S}chool, {U}niv. {A}ntwerp, {A}ntwerp, 1972, 143--316.

\bibitem{DS} J. Demeio, S. Streeter, {\em Weak approximation for del Pezzo surfaces of low degree.} Int. Math. Res. Not. IMRN 2023, no. 13, 11549--11576.

\bibitem{DEN} J. Denef, {\em Proof of a conjecture of {C}olliot-{T}h\'{e}l\`ene and a diophantine excision theorem.} Algebra Number Theory, 13(9):1983--1996, 2019.

\bibitem{DOL} I. Dolgachev, {\em Classical algebraic geometry. A modern view.} Cambridge University Press, Cambridge, 2012.

\bibitem{FvLfull} D. Festi, R. van Luijk, {\em Unirationality of del Pezzo surfaces of degree 2 over finite fields.} Bulletin of the London Mathematical Society, 48(1):135--140, 2016.
		
\bibitem{FvL} D. Festi, R. van Luijk, {\em Unirationality of del Pezzo surfaces of degree 2 over finite fields (extended version).} Preprint, arXiv:1408.0269.

\bibitem{FJ} M. Fried, M. Jarden, {\em Field arithmetic,} Third Edition, revised by Moshe Jarden, Ergebnisse der Mathematik (3) 11, Springer, Heidelberg, 2008.

\bibitem{EGA} A. Grothendieck, {\em \'El\'ements de g\'eom\'etrie alg\'ebrique. IV. \'Etude locale des sch\'emas et des morphismes de sch\'emas. III.} Inst. Hautes Études Sci. Publ. Math. No. 28 (1966).

\bibitem{HAR} R. Hartshorne, {\em Algebraic geometry.} Graduate Texts in Mathematics, No. 52. Springer-Verlag, New York-Heidelberg, 1977.

\bibitem{KOL} J. Koll\'ar, {\em Rational curves on algebraic varieties.} Ergebnisse der Mathematik und ihrer Grenzgebiete. 3. Folge. A Series of Modern Surveys in Mathematics, 32. Springer-Verlag, Berlin, 1996. 

\bibitem{KT} A. Kresch, Y. Tschinkel, {\em On the arithmetic of del Pezzo surfaces of degree 2}, Proc. London Math. Soc. (3) 89 (2004), no. 3, 545–569.

\bibitem{KUW} M. Kuwata, {\em Twenty-eight double tangent lines of a plane quartic curve with an involution and the {M}ordell-{W}eil lattices.} {Comment. Math. Univ. St. Pauli}, (2005) 17--32, 51, 1

\bibitem{LW} S. Lang, A. Weil, {\em Number of points of varieties in finite fields.} {Amer. J. Math.}, 76 (1954), 819--827

\bibitem{LIU} Q. Liu, {\em Algebraic geometry and arithmetic curves.} Oxford Graduate Texts in Mathematics, 6. Oxford Science Publications. Oxford University Press, Oxford, 2002.

\bibitem{MAN70} Y. Manin, {\em Le groupe de Brauer-Grothendieck en g\'eom\'etrie diophantienne.} Actes du Congrès International des Mathématiciens (Nice, 1970), Tome 1, pp. 401--411. Gauthier-Villars, Paris, 1971.

\bibitem{MAN} Y. Manin, {\em Cubic forms. Algebra, geometry, arithmetic,} Second edition. North-Holland Mathematical Library, 4. North-Holland Publishing Co., Amsterdam, 1986.

\bibitem{POO} B. Poonen, {\em Rational points on varieties.} Graduate Studies in Mathematics, 186. American Mathematical Society, Providence, RI, 2017.

\bibitem{SS} P. Salberger, A. Skorobogatov, {\em Weak approximation for surfaces defined by two quadratic forms.} Duke Math. J. 63 (1991), no. 2, 517--536.

\bibitem{STVA} C. Salgado, D. Testa, A. V\'arilly-Alvarado, {\em On the unirationality of del Pezzo surfaces of degree 2.} J. Lond. Math. Soc. (2) 90 (2014), no. 1, 121--139.

\bibitem{SS19} M. Sch\"utt, T. Shioda, {\em Mordell--Weil Lattices.} Ergebnisse der Mathematik und ihrer Grenzgebiete, 3. Folge (A Series of Modern Surveys in Mathematics), vol. 70, Springer, Singapore, 2019.

\bibitem{SER} J.-P. Serre, {\em Topics in Galois theory,} second edition. Research Notes in Mathematics, 1. A K Peters, Ltd., Wellesley, MA, 2008.

\bibitem{SEG} B. Segre, {\em A note on arithmetical properties of cubic surfaces.} J. London Math. Soc 18, (1943). 24--31.

\bibitem{SIL} J. Silverman, {\em The arithmetic of elliptic curves.} Second edition. Graduate Texts in Mathematics, 106. Springer, Dordrecht, 2009.

\bibitem{SD} P. Swinnerton-Dyer, {\em Weak approximation and $R$-equivalence on cubic surfaces.} Rational points on algebraic varieties, 357--404, Progr. Math., 199, Birkhäuser, Basel, 2001.

\bibitem{VAR} A. V\'arilly-Alvarado, {\em Weak approximation on del Pezzo surfaces of degree 1.} Adv. Math. 219 (2008), no. 6, 2123–2145.

\bibitem{SP} Various authors, {\em The Stacks Project.} \url{https://stacks.math.columbia.edu}.
\end{thebibliography}
\end{document}